\documentclass[11pt, twoside, fullpage]{article}
\usepackage{amssymb, amsmath, amsthm, verbatim}
\usepackage[top=30mm, left=23mm, right=23mm, bottom=30mm]{geometry}
\usepackage{inputenc}
\usepackage{fancyhdr}
\usepackage{tikz}
\usepackage{verbatim}
\usetikzlibrary{positioning}
\usepackage{titling}
\usepackage[hyphens]{url}
\usepackage{hyperref}
\usepackage[hyphenbreaks]{breakurl}
\usepackage[T1]{fontenc}
\usepackage{currvita}
\usepackage{xcolor}
\predate{}
\postdate{}
\usepackage{lipsum}
\newtheorem{theorem}{Theorem}
\newtheorem{corollary}[theorem]{Corollary}
\newtheorem{lemma}[theorem]{Lemma}
\newtheorem{proposition}[theorem]{Proposition}

\newtheorem{claim1}{Claim}
\newtheorem{claim2}{Claim}

\newtheorem{conjecture}[theorem]{Conjecture}

\begin{document}
\newcommand{\Addresses}{{
\bigskip
\footnotesize
\medskip

\noindent Maria-Romina~Ivan, \textsc{Department of Pure Mathematics and Mathematical Statistics, Centre for Mathematical Sciences, Wilberforce Road, Cambridge, CB3 0WB, United UK.}\par\noindent\nopagebreak\textit{Email addresses: }\texttt{mri25@dpmms.cam.ac.uk}

\medskip

\noindent Sean~Jaffe, \textsc{Department of Pure Mathematics and Mathematical Statistics, Centre for Mathematical Sciences, Wilberforce Road, Cambridge, CB3 0WB, United UK.}\par\noindent\nopagebreak\textit{Email address: }\texttt{scj47@cam.ac.uk}}}

\pagestyle{fancy}
\fancyhf{}
\fancyhead [LE, RO] {\thepage}
\fancyhead [CE] {MARIA-ROMINA IVAN AND SEAN JAFFE}
\fancyhead [CO] {GLUING POSETS AND THE DICHOTOMY OF POSET SATURATION NUMBERS}
\renewcommand{\headrulewidth}{0pt}
\renewcommand{\l}{\rule{6em}{1pt}\ }
\title{\Large{\textbf{GLUING POSETS AND THE DICHOTOMY OF POSET SATURATION NUMBERS}}}
\author{MARIA-ROMINA IVAN AND SEAN JAFFE}
\date{}
\maketitle
\begin{abstract}
Given a finite poset $\mathcal P$, we say that a family $\mathcal F$ of subsets of $[n]$ is $\mathcal P$-saturated if $\mathcal F$ does not contain an induced copy of $\mathcal P$, but adding any other set to $\mathcal F$ creates an induced copy of $\mathcal P$. The saturation number of $\mathcal P$ is the size of the smallest $\mathcal P$-saturated family with ground set $[n]$. The saturation number for posets is known to exhibit a dichotomy: it is either bounded or it has at least $\sqrt n$ rate of growth. 
Determining which posets have bounded saturation number is a major open problem.

In this paper we consider a `gluing' operation, formed from two finite posets $\mathcal P$ and $\mathcal Q$ by setting all elements of $\mathcal P$ to be below all elements of $\mathcal Q$. We show that (under some mild assumptions) this
operation preserves bounded and unbounded saturation number. This is the first such `new from old' poset construction to be found.
As an application, we show that for any poset $\mathcal P$ one may add at most 3 elements to $\mathcal P$ to obtain a poset whose saturation number growth is at most linear: this may be viewed as a step towards the
other major open problem in the area, namely the conjecture that every finite poset has this growth at most linear.

We also consider the poset equivalent of weak saturation for
graphs: for each finite poset $\mathcal P$, we determine exactly the minimum size of a percolating family for $\mathcal P$.
\end{abstract}
\section{Introduction}
We say that a poset $(\mathcal Q,\preceq)$ contains an \textit{induced copy} of a poset $(\mathcal P,\preceq')$ if there exists an injective order-preserving function $f:\mathcal P\rightarrow\mathcal Q$ such that $(f(\mathcal P),\preceq)$ is isomorphic to $(\mathcal P,\preceq')$. We denote by $\mathcal P([n])$ the power set of $[n]=\{1,2,\dots,n\}$. More generally, for any finite set $S$, we denote by $\mathcal P(S)$ the power set of $S$. We define the \textit{$n$-hypercube}, denoted by $Q_n$ to be the poset formed by equipping $\mathcal P([n])$ with the partial order induced by inclusion.

If $\mathcal P$ is a finite poset and $\mathcal F$ is a family of subsets of $[n]$, we say that $\mathcal F$ is $\mathcal P$-\textit{saturated} if $\mathcal F$ does not contain an induced copy of $\mathcal P$, and for any $S\notin\mathcal F$, the family $\mathcal F\cup\{S\}$ contains an induced copy of $\mathcal P$. The smallest size of a $\mathcal P$-saturated family of subsets of $[n]$ is called the \textit{induced saturated number}, denoted by $\text{sat}^*(n,\mathcal P)$.

It has been shown that the growth of $\text{sat}^*(n,\mathcal P)$ has a dichotomy. Keszegh, Lemons, Martin, P{\'a}lv{\"o}lgyi and Patk{\'o}s \cite{keszegh2021induced} proved that for any poset the induced saturated number is either bounded or at least $\log_2(n)$. They also conjectured that in fact $\text{sat}^*(n,\mathcal P)$ is either bounded, or at least $n+1$. Later, Freschi, Piga, Sharifzadeh and Treglown \cite{freschi2023induced} improved this result by replacing $\log_2 (n)$ with $2\sqrt{n}$. There is no known poset $\mathcal P$ for which $\text{sat}^*(n,\mathcal P)=\omega(n)$, and it is in fact believed that for any poset, the saturation number is either constant or it grows linearly.

But what determines if a poset does or does not have unbounded saturation number? For an example showing lack of monotonicity we mention that the posets $C_2$ and $3C_2$, displayed below, have bounded saturation number \cite{keszegh2021induced}, while $2C_2$ \cite{keszegh2021induced} has unbounded saturation number.
\begin{center}
\includegraphics[width=6.5cm]{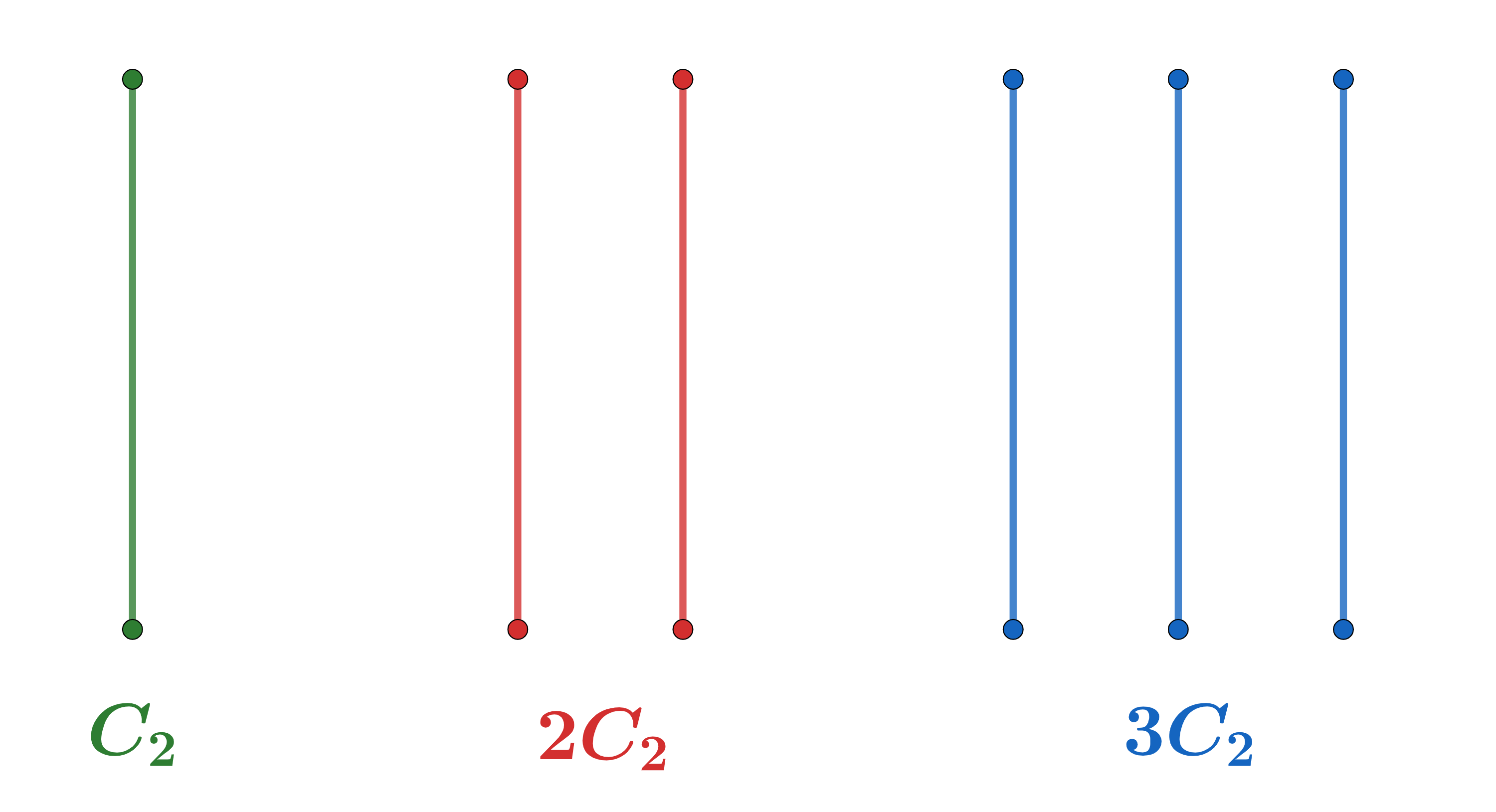}
\end{center}
We say that a family $\mathcal F\subseteq\mathcal P([n])$ \textit{separates the ground set} if, for any two distinct elements $i,j\in[n]$, there exists a set in $\mathcal F$ that contains one but not the other. It was shown in \cite{keszegh2021induced,ferrara2017saturation} that the existence of an $n_0$ and a $\mathcal P$-saturated family in $\mathcal P([n_0])$ that does not separate the elements of the ground set, implies that $\mathcal P$ has bounded saturation number. Conversely, if a family separates the ground set $[n]$, it has size at least $\log_2(n)$. Therefore, $\mathcal P$ has bounded saturation number if and only if there exists a saturated family that does not separate the ground set. However, this tells us nothing about the intrinsic properties a poset must have in order for such a family to exist. In fact, this question is very much open as not even good guesses about what these properties might be exist as of now.

Partial progress towards answering this question was made in \cite{ferrara2017saturation} where it was shown that if a poset has the \textit{unique cover twin property (UCTP)}, then it has unbounded saturation number. In a poset, $y$ covers $x$ if $x<y$ and there is no $z$ with $x<z<y$. We say that a poset has the UCTP if, for any two elements $p_1$ and $p_2$ of the poset, if $p_2$ covers $p_1$, then there exists $p_3$, called their \textit{twin}, that is comparable with exactly one of $p_1$ and $p_2$. We mention that in the literature there is some variation when it comes to the definitions of the UCTP and that of the twin. They are equivalent as far as the results are concerned. An extension of this result was proved in \cite{keszegh2021induced}, where it was shown that attaching a chain above any poset with the UCTP gives a poset with unbounded saturation number still. It was also shown by Freschi, Piga, Sharifzadeh and Treglown \cite{freschi2023induced} that any poset with legs has a saturation number that is at least linear, and therefore unbounded. A poset $\mathcal P$ is said to \textit{have legs} if there exist distinct elements $l_1, l_2, h\in\mathcal P$ such that $l_1, l_2$ are incomparable and less than $h$, and any other element of the poset is greater than $h$. This was further generalised by Liu \cite{liu2025induced} who removed the intermediary point $h$. Sadly, these conditions are not exhaustive as the poset $2C_2$, which does not have the UCTP (or UCTP with top chain) and does not have legs, was shown to have unbounded saturation number \cite{keszegh2021induced}. It is worth pointing out that until now this was the only such example. 
For any finite poset $\mathcal P$, let $\dot{\mathcal P}$ denote the poset obtained from $\mathcal P$ by adding a new element that is greater than all the rest. Motivated by the observation that if $\mathcal P$ is a poset without a unique maximal element, then any $\mathcal P$-saturated family must contain the full set, meaning that whenever a copy of $\mathcal P$ is created a copy of $\dot{\mathcal P}$ is created too, the following conjecture was made.
\begin{conjecture}[\cite{keszegh2021induced}]
\label{maxelement}
Let $\mathcal P$ be a finite poset. Then $\text{sat}^*(n,\mathcal P)$ is bounded if and only if $\text{sat}^*(n,\dot{\mathcal P})$ is bounded.   
\end{conjecture}
We mention that forbidden poset problems is a new and rapidly growing area of study in combinatorics. Saturation for posets was introduced by Gerbner, Keszegh, Lemons, Palmer, P{\'a}lv{\"o}lgyi and Patk{\'o}s \cite{gerbner2013saturating}, although this was not for \textit{induced} saturation. Induced poset saturation was first introduced in 2017 by Ferrara, Kay, Kramer, Martin, Reiniger, Smith and Sullivan \cite{ferrara2017saturation}. As of now, there are very few posets for which the growth order of the saturation number is known, and far fewer for which the exact number is known. Some posets for which it is known that the saturation number is linear include $\mathcal V_2$, $\Lambda_2$ \cite{ferrara2017saturation}, the butterfly \cite{ivan2020saturationbutterflyposet, keszegh2021induced}, and the antichain \cite{keszegh2021induced, dhankovic2023saturation, bastide2024exact}. One of the most notorious posets for which the saturation number is unknown is the diamond, the 2-dimensional hypercube, for which the current lower bound is $(4-o(1))\sqrt n$ \cite{ivan2021minimal}, whilst the upper bound is $n+1$ \cite{ferrara2017saturation}. For a nice introduction to the area, we refer the reader to the textbook of Gerbner and Patk{\'o}s \cite{gerbner2018extremal}.

In this paper we discuss the \textit{gluing} operation of posets which turns out to preserve the bounded or unbounded nature of the saturation number of the starting posets. More precisely, for two finite posets $\mathcal P_1$ and $\mathcal P_2$, we define $\mathcal P_2*\mathcal P_1$ to be the poset comprised of an induced copy of $\mathcal P_2$ completely above an induced copy of $\mathcal P_1$, i.e all elements of $\mathcal P_2$ are strictly greater than all the elements of $\mathcal P_1$. For completeness, if $\mathcal P_2$ is the empty poset, then we define $\mathcal P_2*\mathcal P_1$ to be $\mathcal P_1$, and if $\mathcal P_1$ is the empty poset we define $\mathcal P_2*\mathcal P_1$ to be $\mathcal P_2$. Trivially, this operation is associative. It turns out that $\text{sat}^*(n,\mathcal P_2*\mathcal{P}_1$) is controlled, at least in one direction, by $\text{sat}^*(n,\mathcal P_2 * \bullet*\mathcal P_1)$, where $\bullet$ represents the one element poset. The Hasse diagrams of these two posets is illustrated below.
\begin{center}
\includegraphics[width=11.5cm]{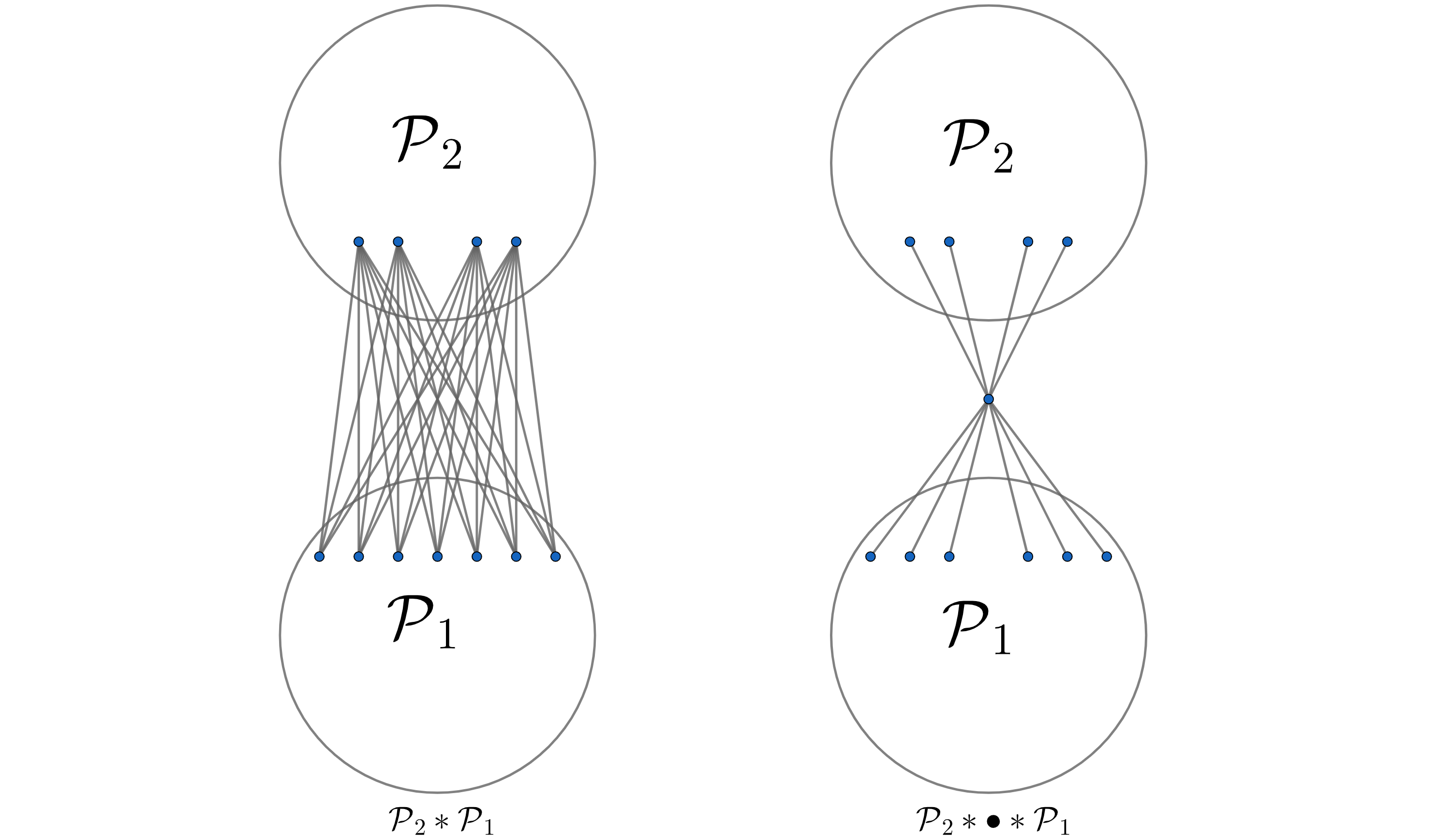}
\end{center}
We also say that a poset has a unique maximal (minimal) element if it contains an element that is greater (smaller) than all the other elements of the poset.

The plan of the paper is as follows. In Section 2 we show that if one of $\mathcal P_1$ or $\mathcal P_2$ has the UCTP, then $\mathcal P_2*\mathcal P_1$ has unbounded saturation number. This generates a plethora of posets that do not have the UCTP but do have unbounded saturation number, for example $\mathcal P*\mathcal A_3$ where $\mathcal P$ is any poset and $\mathcal A_3$ is the antichain of size 3. Furthermore, this extends the result about posets with legs having unbounded saturation number, as any such poset is in fact $\mathcal P*\bullet*\mathcal A_2$, where $\mathcal{A}_2$ is the antichain of size 2. We also discuss in this section some surprising properties of this gluing operation which may be of independent interest, especially towards Conjecture~\ref{maxelement}. Most notably, in Proposition~\ref{propositionmax} we show that the saturation number of $\mathcal P_2*\mathcal P_1$ is at least the saturation number of both $\mathcal P_2*\bullet$ and $\bullet*\mathcal P_1$, where $\mathcal P_1$ and $\mathcal P_2$ are non-empty posets without a unique maximal and a unique minimal element, respectively.

In Section 3 we study \textit{special} posets. We say that a poset is special if it is comprised of a single element, and another unrelated poset with a unique maximal element and a unique minimal element. We show that if $\mathcal P_1$ and $\mathcal P_2$ are two special posets, then the saturation number of $\mathcal P_2*\mathcal P_1$ is bounded above by a linear function of $\text{sat}^*(n,\mathcal P_1)$ and $\text{sat}^*(n, \mathcal P_2)$, which is otherwise independent of $n$. In particular, gluing two special posets with bounded saturation number will result in another poset with bounded saturation number. We also give linear upper bounds for all special posets. This gives us an interesting corollary. Given any finite poset $\mathcal P$, we can embed it in a special poset $\mathcal P^*$ by adding, if necessary, a unique minimal, a unique maximal and an unrelated element. Thus, $\mathcal P$ is a sub-poset of $\mathcal P^*$ which has saturation number at most linear, and $|\mathcal P^*|\leq |\mathcal P|+3$, getting close to the conjecture that the saturation number of any poset is at most linear.

In Section 4 we focus entirely on a subclass of posets generated by repeatedly gluing antichains, i.e. multilayered complete posets. Under the assumption that there are no two consecutive layers of size 1, we establish a linear upper bound for all such posets, extending a result of Liu \cite{liu2025induced} who established such a bound for bipartite complete posets. Furthermore, since our construction is of a completely different nature, we also get a better multiplicative constant for the saturation number of bipartite posets.

Finally, in Section 5 we discuss the poset equivalent notion of weak saturation for graphs. In other words, for a finite poset $\mathcal P$, we want to find the smallest family $\mathcal F\subseteq\mathcal P([n])$ such that there exists a way of adding all the sets outside of $\mathcal F$, one by one, such that each time we create a new induced copy of $\mathcal P$. We call such an $\mathcal F$ a \textit{$\mathcal P$-percolating} family. We call the size of the smallest $\mathcal P$-percolating family \textit{the percolation number} of $\mathcal P$, denoted by $\text{sat}_p(n,\mathcal P)$. We show that $\text{sat}_p(n,\mathcal P)$ is $|\mathcal P|+1$ if $\mathcal P$ does not have a unique maximal or a unique minimal element, $|\mathcal P|-1$ if $\mathcal P$ has both a unique maximal and a unique minimal element, and $|\mathcal P|$ otherwise, answering a question from \cite{workshop}. Moreover, we show that, with the exception of the set $[n]$ which must be present if $\mathcal P$ does not have a unique maximal element, the $\mathcal P$-percolating family can be taken to be fixed, i.e. the same for all $n$ sufficiently large.

Throughout the paper our notation is standard. For a finite set $S$, we denote by $\binom{S}{k}$ the collection of all subsets of $S$ of size $k$. When $S=[n]$ this is also known as the $k^{\text{th}}$ layer of $Q_n$. We also denote by $\binom{S}{\leq k}$ the collection of all subsets of $S$ of size at most $k$.

\section{Families of non-UCTP posets with unbounded $\text{sat}^*(n,\mathcal P)$}
The main aim of this section is to show that gluing operation of two posets gives rise to a poset with unbounded saturation number, as long as  at least one of them has the UCTP. The proof has two main ingredients. The first one is Proposition~\ref{glueinglemma}, an intermediary step, in which we show that the saturation number is unbounded if we add a middle point strictly between the two posets. The proof relies on the fact that a poset has unbounded saturation number if and only if any saturated family separates the ground set. 

The second part of the proof is Proposition~\ref{poset-point-poset} which is a general statement that shows the close relationship between gluing two posets and gluing them via a single point in the middle. More precisely, if $\mathcal P_2$ and $\mathcal P_1$ are two posets without a unique minimal element or a unique maximal element, respectively, then a $\mathcal P_2*\mathcal P_1$-saturated family is also $\mathcal P_2*\bullet*\mathcal P_1$-saturated.

We start with a simple lemma.
\begin{lemma}\label{UCTP} If $\mathcal Q_1$ is a poset with the UCTP and a unique maximal (minimal) element, then removing that element results in a poset with the UCTP and without a maximal (minimal) element. Furthermore, if $\mathcal Q_2$ is a poset with the UCTP that does not have a unique maximal (minimal) element, then $\bullet*\mathcal Q_2$ ($\mathcal Q_2*\bullet)$ has the UCTP.
\end{lemma}
\begin{proof}
For the first part, let $q$ be the unique maximal element of $\mathcal Q_1$. For any $x$ that is covered by $y$ in $\mathcal Q_1\setminus\{q\}$, and hence in $\mathcal Q_1$, their twin could not be $p$ since it is comparable to both. Thus, $\mathcal Q_1\setminus\{q\}$ has the UCTP. Furthermore, $\mathcal{Q}_1 \setminus \{q\}$ cannot have a unique maximal element, say, $r$, since $q$ would be the unique cover of $r$ in $\mathcal{Q}_1$, and $q$ would not have a twin, as every element of $\mathcal{Q}_1 \setminus \{q, r\}$ is comparable to both. By taking complements, the same analysis is true for unique minimal elements.

The second claim of the lemma follows from the observation that the only elements of $Q_2$ covered (covering) by the new unique maximal (minimal) element are the maximal (minimal) elements of $\mathcal Q_2$. Since $\mathcal Q_2$ does not have a unique maximal (minimal) element, they are incomparable, hence each other's twins.
\end{proof}
\begin{proposition}\label{glueinglemma}
Suppose $\mathcal P_1$ has the UCTP, $|\mathcal P_1|\geq 2$, and $\mathcal P_2$ is any poset. Then $sat^*(n,\mathcal P_2 * \bullet*\mathcal P_1)\rightarrow\infty$ as $n\rightarrow\infty$.
\end{proposition}

\begin{proof}
Suppose  first that $\mathcal P_1$ has a unique maximal element, call it $p$. Then $\mathcal P_2 *\bullet*\mathcal P_1$ can be thought of as $(\mathcal P_2*\bullet)*p*(\mathcal P_1\setminus\{p\})$. By Lemma~\ref{UCTP} we know that $\mathcal P_1\setminus\{p\}$ still has the UCTP. Therefore, without loss of generality, we may assume that $\mathcal P_1$ does not have a unique maximal element.

As explained in the introduction, it suffices to show that if $\mathcal F\subseteq \mathcal P([n])$ is a $\mathcal P_2 *\bullet * \mathcal P_1$-saturated family, then $\mathcal F$ separates the points of the ground set $[n]$. In other words, for any two distinct $i,j\in[n]$, there exists a set in $\mathcal F$ that contains one but not the other.

Let $\mathcal{F}$ be a $\mathcal P_2*\bullet*\mathcal P_1$-saturated family which does not separate $i$ and $j$ for some $i,j\in [n]$. We can therefore split the family as follows: $\mathcal F=\mathcal F_0\cup\mathcal F_2$, where $\mathcal F_0=\{A\in\mathcal F: A\cap\{i,j\}=\emptyset\}$ and $\mathcal F_2=\{A\in\mathcal F:\{i,j\}\subseteq A\}$. After relabeling the ground set, we may assume for simplicity that $i=1$ and $j=2$.

Consider now the singleton $\{1\}$ which, by assumption, is not in $\mathcal F$. Thus $\mathcal F\cup\{\{1\}\}$ contains an induced copy of $\mathcal P_2*\bullet*\mathcal P_1$, which must contain the singleton $\{1\}$. Examine the $\mathcal{P}_2 * \bullet$ part of this copy. We observe that any element in this part is strictly greater than at least two other sets, namely those in the $\mathcal{P}_1$ part. As such, we cannot have the singleton $\{1\}$ in the $\mathcal P_2*\bullet$ part of the copy, so it must be in the $\mathcal P_1$ part. Therefore, all elements of the $\mathcal P_2*\bullet$ part are in $\mathcal F$ and contain 1. Hence, $\mathcal F_2$ contains an induced copy of $\mathcal P_2*\bullet$.

Let $S$ be the set that is less than all of $\mathcal P_2$ in this copy of $\mathcal P_2*\bullet$ if $\mathcal F_2$. Furthermore, let $|S|$ be minimal over all sets satisfying this property. In other words, if $A$ represents the minimal element of an induced copy of $\mathcal P_2*\bullet$ in $\mathcal F_2$, then $|A|\geq |S|$.

\begin{center}
\includegraphics[width=5cm]{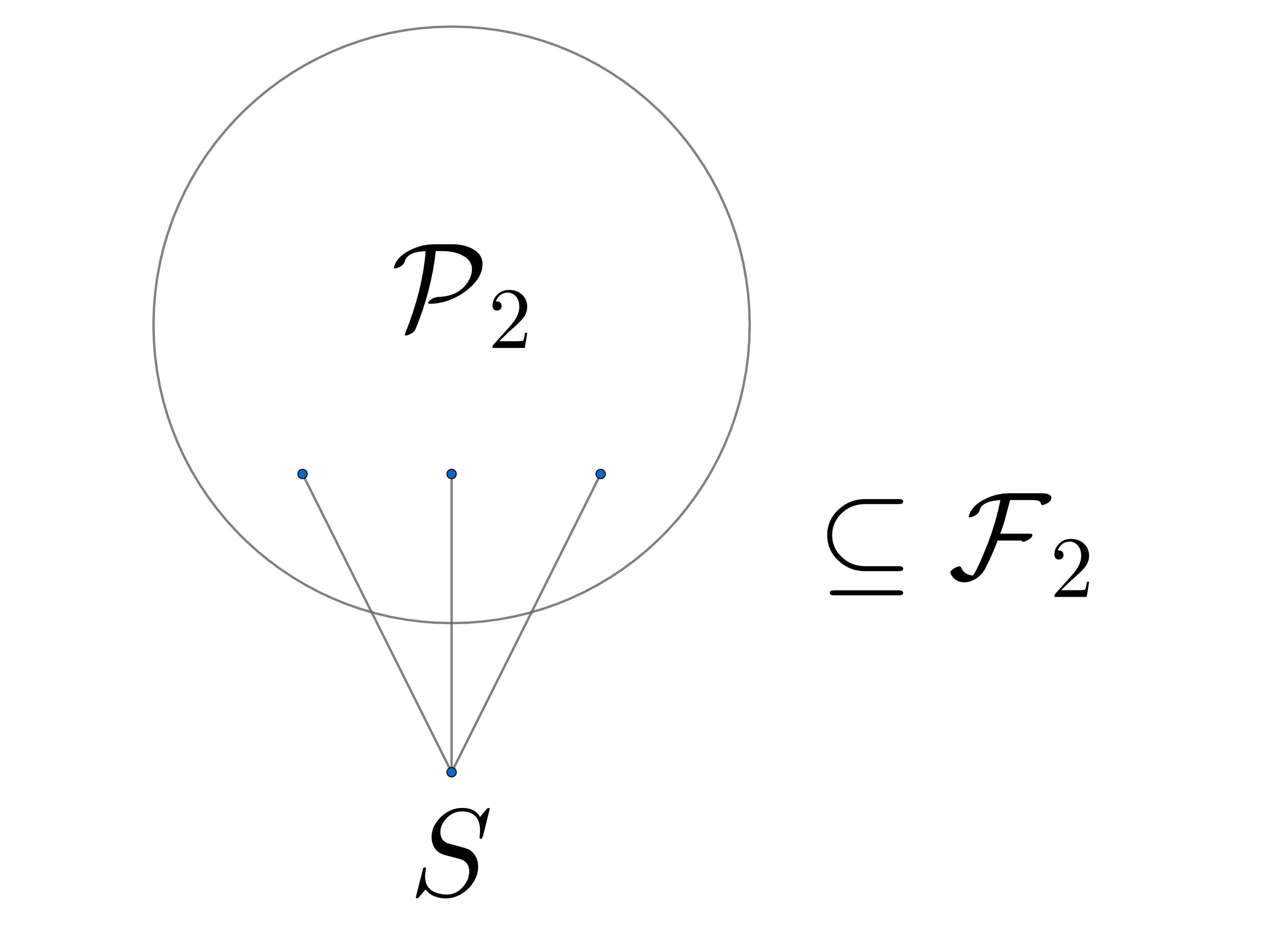}
\end{center}

By assumption, $S=S_0\cup\{1,2\}\in\mathcal F$ and $S_0\cup\{1\}\notin\mathcal F$. Thus, $\mathcal F \cup \{S_0\cup\{1\}\}$ must contain an induced copy of $\mathcal P_2*\bullet*\mathcal P_1$, and $S_0\cup\{1\}$ must be part of this copy.

If $S_0\cup\{1\}$ belongs to the $\mathcal P_2*\bullet$ part of this copy, then its $\mathcal P_1$ copy is in $\mathcal F_0$ and it is less than $S_0\cup\{1\}$, hence also less than $S$. However, $S$ is the minimal element of a copy of $\mathcal P_2*\bullet$ in $\mathcal F$, and so, putting this copy on top of the above copy of $\mathcal P_1$ we obtain an induced copy of $\mathcal P_2*\bullet*\mathcal P_1$ in $\mathcal F$, a contradiction. Therefore, $S_0\cup\{1\}$ is part of the $\mathcal P_1$ part.

Now, observe that if $S$ is not an element of this copy of $\mathcal P_2*\bullet*\mathcal P_1$, then we may replace $S_0\cup\{1\}$ with $S$, and all relations are preserved, since $\mathcal F$ does not separate 1 and 2, thus yielding an induced copy of $\mathcal P_2*\bullet*\mathcal P_1$ in $\mathcal F$, which leads to a contradiction. Therefore both $S$ and $S_0\cup\{1\}$ must be elements of this copy of $\mathcal P_2*\bullet*\mathcal P_1$. Moreover, since $S_0\cup\{1\}$ is in the $\mathcal P_1$ part of the copy, and $S=S_0\cup\{1,2\}$, $S$ must be in the $\bullet*\mathcal P_1$ part of the copy. 
\begin{center}
\includegraphics[width=10cm]{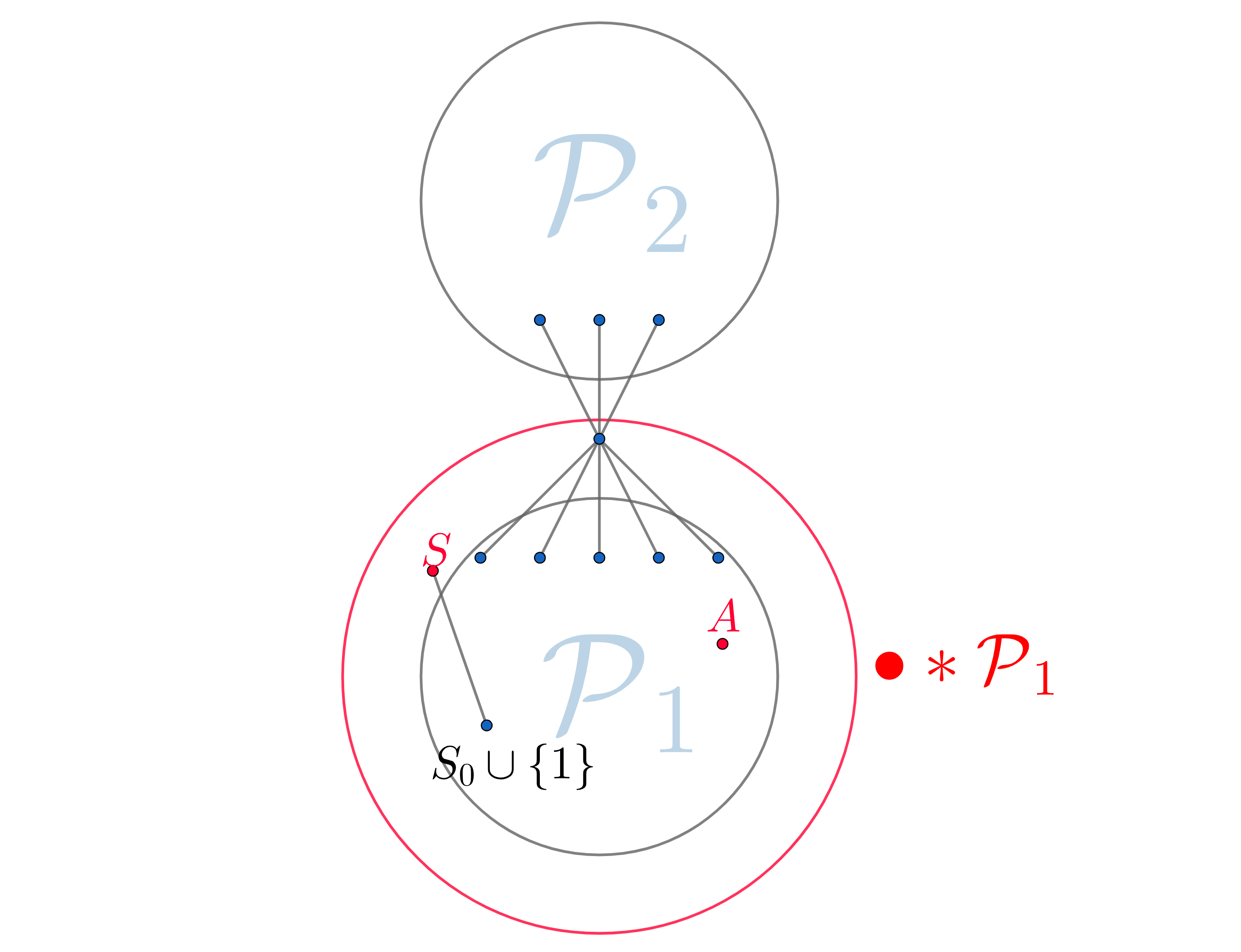}
\end{center}
But now we see that the poset $\bullet*\mathcal P_1$ has the UCTP (since $\mathcal P_1$ does not have a unique maximal element). As $S$ covers $S_0\cup\{1\}$, there exists a third set $A$ in the $\bullet*\mathcal P_1$ part of this copy of $\mathcal P_2*\bullet*\mathcal P_1$ (which therefore is an element of $\mathcal F$) such that either $A$ covers $S_0\cup \{1\}$, or $A$ and $S_0\cup \{1\}$ are twins. This means that either $A$ and $S$ are incomparable and $S_0\cup\{1\}\subsetneq A$, or  $A\subsetneq S$ and $S_0\cup\{1\}$ is incomparable to $A$.

In the first case, since $A$ and $S$ are incomparable and $S\setminus\{2\}\subseteq A$, we must have $1\in A$ and $2\notin A$, a contradiction as $A$ would separate 1 and 2. In the second case, as $A\subsetneq S$ and $A$ and $S\setminus\{2\}$ are incomparable, we deduce that $2\in A$, so $A\in \mathcal F_2$. However, $|A|<|S|$, and replacing $S$ with $A$ in a copy of $\mathcal P_2*\bullet$, where $S$ is the minimal element, gives a copy of $\mathcal P_2*\bullet$ whose minimal element is $A$. This contradicts the minimality of $|S|$, and thus completes the proof.
\end{proof}
We now move on and establish the following relation between $\text{sat}^*(n,\mathcal P_2*\mathcal P_1)$ and $\text{sat}^*(n,\mathcal P_2*\bullet*\mathcal P_1)$, where $\mathcal P_2$ is assumed to not have a unique minimal element, and $\mathcal P_1$ is assumed to not have a unique maximal element.
\begin{proposition}\label{poset-point-poset} Let $\mathcal P_1$ and $\mathcal P_2$ be non-empty posets such that $\mathcal P_1$ does not have a unique maximal element and $\mathcal P_2$ does not have a unique minimal element. Let $\mathcal F$ be a $\mathcal P_2*\mathcal P_1$-saturated family with ground set $[n]$. Then $\mathcal F$ is also a $\mathcal P_2*\bullet*\mathcal P_1$-saturated family with ground set $[n]$.
\end{proposition}

\begin{proof}
We call $\mathcal P_1$ the `lower' part, and $\mathcal P_2$ the `upper' part of $\mathcal P_2*\mathcal P_1$ respectively.

Let $\mathcal F$ be a $\mathcal P_2*\mathcal P_1$-saturated family. Since it is $\mathcal P_2*\mathcal P_1$-free, it is also $\mathcal P_2*\bullet*\mathcal P_1$-free. Suppose that $\mathcal F$ is not $\mathcal P_2*\bullet*\mathcal P_1$-saturated. Then there exists some $A\notin\mathcal F$ such that $\mathcal F\cup \{A\}$ does not contain a copy of $\mathcal P_1*\bullet*\mathcal P_2$. However, by assumption, $\mathcal F \cup\{A\}$ contains an induced copy of $\mathcal P_2*\mathcal P_1$, and $A$ must be an element of any such copy.

Furthermore, we observe that no $A\notin\mathcal F$ can be both a part of the lower part and upper part of some (possibly distinct) copies of $\mathcal P_2*\mathcal P_1$ in $\mathcal F$. This is because if $A$ is in the upper part of $\widehat{\mathcal P_2}*\widehat{\mathcal P_1}$ and in the lower part of $\widetilde{\mathcal P_2}*\widetilde{\mathcal P_1}$ (where $\widehat{\mathcal P_2}*\widehat{\mathcal P_1}$ and $\widetilde{\mathcal P_2}*\widetilde{\mathcal P_1}$ are both copies of $\mathcal P_2*\mathcal P_1$ in $\mathcal F\cup \{A\}$), then $\widetilde{\mathcal P}_1,\widehat{\mathcal P_2}\subseteq \mathcal{F}$. Furthermore, $A$ is greater than every element of $\widetilde{\mathcal P_1}$ and smaller than every element of $\widehat{\mathcal P_2}$. Therefore, every element of $\widetilde{\mathcal P_1}$ is smaller than every element of $\widehat{\mathcal P_2}$, so $\widehat{\mathcal{P}_2} \cup \widetilde{\mathcal{P}_1}$ is an induced copy of $\mathcal{P}_2 *\mathcal{P}_1$ in $\mathcal{F}$, a contradiction.

Since the structure is invariant under taking complements, we may assume, without loss of generality, that $A$ may only appear in the lower part of a copy of $\mathcal P_2*\mathcal P_1$. Out of all the copies of $\mathcal P_2*\mathcal P_1$ that $A$ creates when added to $\mathcal F$, we look at their upper part (which does not contain $A$ by assumption) and pick the one with maximal intersection. More formally, let $\widehat{\mathcal P_2}*\widehat{\mathcal P_1}$ be a copy of $\mathcal P_2*\mathcal P_1$ in $\mathcal F \cup \{A\}$ with $|\cap_{B \in \widehat{\mathcal P_2}}B|\) maximal. Note also that $A\in\widehat{\mathcal P_1}$.

Let $S = \cap_{B \in \widehat{\mathcal P_2}}B$. Then $S$ is smaller than or equal to every element of $\widehat{\mathcal P_2}$ and greater than or equal to every element of $\widehat{\mathcal P_1}$. Since $\widehat{\mathcal{P}_1}$ does not have a unique maximal element, and $\widehat{\mathcal{P}_2}$ does not have a unique minimal element, $S \not \in \widehat{\mathcal{P}_2}*\widehat{\mathcal{P}_1}$. Therefore, $\{S\} \cup \widehat{\mathcal{P}_2} *\widehat{\mathcal{P}_1}$ is an induced copy of $\mathcal P_2*\bullet*\mathcal P_1$ in $\mathcal F\cup\{A,S\}$. Thus $S\notin\mathcal F$, so $\mathcal{F}\cup\{S\}$ contains an induced copy of $\mathcal P_2*\mathcal P_1$, which must contain $S$. For clarity, we will simply refer to this copy as $\mathcal P_2*\mathcal P_1$.

First, we observe that $S$ cannot be in $\mathcal P_2$, as this would mean that $S$ is greater than every element of $\mathcal P_1$. But $S$ is smaller than every element of $\widehat{\mathcal P_2}$ by construction. Therefore, $\widehat{\mathcal P_2}\cup\mathcal P_1$ is an induced copy of $\mathcal P_2*\mathcal P_1$ in $\mathcal F$, a contradiction. So $S\in\mathcal P_1$.

Since $S$ is less than every element of $\mathcal{P}_2$ and greater than every element of $\widehat{\mathcal{P}_1}$, we have that every element of $\widehat{\mathcal{P}_1}$ is less than every element of $\mathcal{P}_2$. Therefore, $\widehat{\mathcal{P}_1}\cup \mathcal{P}_2$ is an induced copy of $\mathcal{P}_2*\mathcal{P}_1$ in $\mathcal{F}\cup \{A\}$. By maximality of $|S|$, we have that $|\cap_{B \in \mathcal{P}_2}B| \leq |S|$. However, since every element of $\mathcal{P}_2$ is greater than every element of $\mathcal{P}_1$, in particular $S$, we must have that $|\cap_{B \in \mathcal{P}_2}B|=|S|$. But this is a contradiction, as this would imply that $S$ is the unique maximal element of $\mathcal P_1$ (and $\mathcal{P}_1$ does not have a unique maximal element). Therefore $\mathcal F$ is $\mathcal P_2*\bullet*\mathcal P_1$-saturated.
\end{proof}
We are now ready to prove the main result of this section.
\begin{theorem}\label{glueingthm}
Let $\mathcal P_1$ and $\mathcal P_2$ be two finite posets such that at least one of them has the UCTP and at least two elements. Then $\text{sat}^*(n,\mathcal P_2*\mathcal P_1)\rightarrow \infty$ as $n\rightarrow\infty$. 
\end{theorem}
\begin{proof}
Since all properties are preserved by looking at the reversed posets, i.e. reversing all the relations and preserving all non-relations, we may assume without loss of generality that $\mathcal{P}_1$ has the UCTP and at least two elements. Furthermore, as noted in Lemma~\ref{UCTP}, if a poset has the UCTP and a unique maximal element, removing this maximal element results in a poset which also has the UCTP. Therefore we may assume that $\mathcal P_1$ does not have a unique maximal element, and consequently that $\mathcal P_2$ does not have a unique minimal element, as otherwise we would be done by Proposition~\ref{glueinglemma}. Thus, by Proposition~\ref{poset-point-poset}, any $\mathcal P_2*\mathcal P_1$-saturated set is also $\mathcal P_2*\bullet*\mathcal P_1$-saturated, which implies $\text{sat}^*(n,\mathcal P_2*\mathcal P_1)\geq \text{sat}^*(n,\mathcal P_2*\bullet*\mathcal P_1)$, which is unbounded, again by Proposition~\ref{glueinglemma}.
\end{proof}
We end the section with the following proposition, which tells us that, if $\mathcal P_2$ and $\mathcal P_1$ do not have a unique minimal element or a unique maximal element, respectively, then the saturation number of $\mathcal P_1*\mathcal P_2$ controls both the saturation number of $\mathcal P_2*\bullet$ and that of $\bullet*\mathcal P_1$. 
\begin{proposition}
\label{propositionmax}
Let $\mathcal P_1$ and $\mathcal P_2$ be any non-empty posets such that $\mathcal P_1$ does not have a unique maximal element and $\mathcal P_2$ does not have a unique minimal element. Then $\text{sat}^*(n,\mathcal P_2*\mathcal P_1)\geq\max\{\text{sat}^*(n,\mathcal P_2*\bullet), \text{sat}^*(n,\bullet*\mathcal P_1)\}$.
\begin{proof}
Let $\mathcal F$ be a $\mathcal P_2*\mathcal P_1$-saturated family with ground set $[n]$. By Proposition~\ref{poset-point-poset}, it is also $\mathcal P_2*\bullet*\mathcal P_1$-saturated.

Define $\mathcal A=\{A:A\text{ is below a copy of }\mathcal P_2\text{ in } \mathcal F\}$ and $\mathcal B=\{A:A\text{ is above a copy of }\mathcal P_1\text{ in } \mathcal F\}$. It is easy to see that a set $A$ cannot be both in $\mathcal A$ and $\mathcal B$, as it would be above a copy of $\mathcal P_1$ and below a copy of $\mathcal P_2$, giving a copy of $\mathcal P_2*\mathcal P_1$ in $\mathcal F$, a contradiction. Therefore, $\mathcal A\cap\mathcal B=\emptyset$. Moreover, for any set $A\notin \mathcal{F}$, there must be a copy of $\mathcal P_2*\mathcal P_1$ in $\mathcal F\cup\{A\}$, with $A$ being an element of that copy. Thus, depending on where $A$ is in that copy (upper or lower part), $A$ is either in $\mathcal A$ or in $\mathcal B$. Hence, $\mathcal F\cup\mathcal A\cup\mathcal B=\mathcal P([n])$.

We will show that $\mathcal F\setminus\mathcal B$ is $\bullet*\mathcal P_1$-saturated. This will imply that $|\mathcal F|\geq|\mathcal F\setminus\mathcal B|\geq\text{sat}^*(n,\bullet*\mathcal P_1)$, and consequently $\text{sat}^*(n,\mathcal P_2*\mathcal P_1)\geq\text{sat}^*(n,\bullet*\mathcal P_1)$. By symmetry, we will also have that $\mathcal F\setminus\mathcal A$ is $\mathcal P_2*\bullet$-saturated, which will finish the proof.

First, we note that $\mathcal{F} \setminus \mathcal{B}$ does not contain a copy of $\bullet * \mathcal{P}_1$, because if it did, its maximal element would be above a copy of $\mathcal{P}_1$, and thus in $\mathcal{B}$, a contradiction.

Now, let $A\notin\mathcal F\setminus\mathcal B$. Then either $A\in\mathcal A\setminus\mathcal F$, or $A\in\mathcal B$.

If $A\in\mathcal A\setminus \mathcal F$, then $\mathcal F\cup\{A\}$ contains a copy of $\mathcal P_2*\bullet*\mathcal P_1$ in which $A$ must appear. Note that $A$ must be in the $\mathcal P_1$ (the lower) part of this copy, as it is in $\mathcal A$ (otherwise it would be above a copy of $\mathcal P_1$ in $\mathcal F$, and thus would be an element of $\mathcal B$). But we are now done, since we have obtained a copy of $\bullet*\mathcal P_1$ which is in $\mathcal F\cup\{A\}$, and is below a copy of $\mathcal P_2$, thus it is in $\mathcal F\cap\mathcal A\cup\{A\}\subseteq\mathcal F\setminus \mathcal B\cup\{A\}$.

If $A\in\mathcal B$, then $A$ is above a copy of $\mathcal P_1$ in $\mathcal F$. For all such copies, we take the one with $|\cup_{X\in\mathcal P_1}X|$ minimal. For convenience, call it $\mathcal P_1$. To finish the proof, it is enough to show that $\mathcal P_1$ is in $\mathcal F\setminus\mathcal B$. We already know that all its elements are in $\mathcal F$, so assume that there exists $Y\in\mathcal P_1$ such that $Y\in\mathcal B$. By definition, $Y$ is above another copy of $\mathcal P_1$, call it $\widehat{\mathcal P_1}$ (which is also below $A$ by transitivity). By minimality, we must have $|\cup_{X\in\mathcal P_1}X|\leq|\cup_{X\in\widehat{\mathcal P_1}}X|$. However, $|\cup_{X\in\widehat{P_1}}X|\leq|Y|<|\cup_{X\in\mathcal P_1}X|$, where the last inequality is strict because $\mathcal P_1$ does not have a unique maximal element. This is a contradiction, so $\mathcal F\setminus\mathcal B$ is $\bullet*\mathcal P_1$-saturated.
\end{proof}
\end{proposition}
\section{Gluing special posets}
We now turn our attention to a certain class of posets: we call a poset $\mathcal P$ \textit{special} if its Hasse diagram is comprised of a poset $\mathcal P'$ with both a unique minimal and a unique maximal element, and another single element with no relations to $\mathcal P'$, as illustrated in the diagram below.
\begin{center}
    \includegraphics[width=5cm]{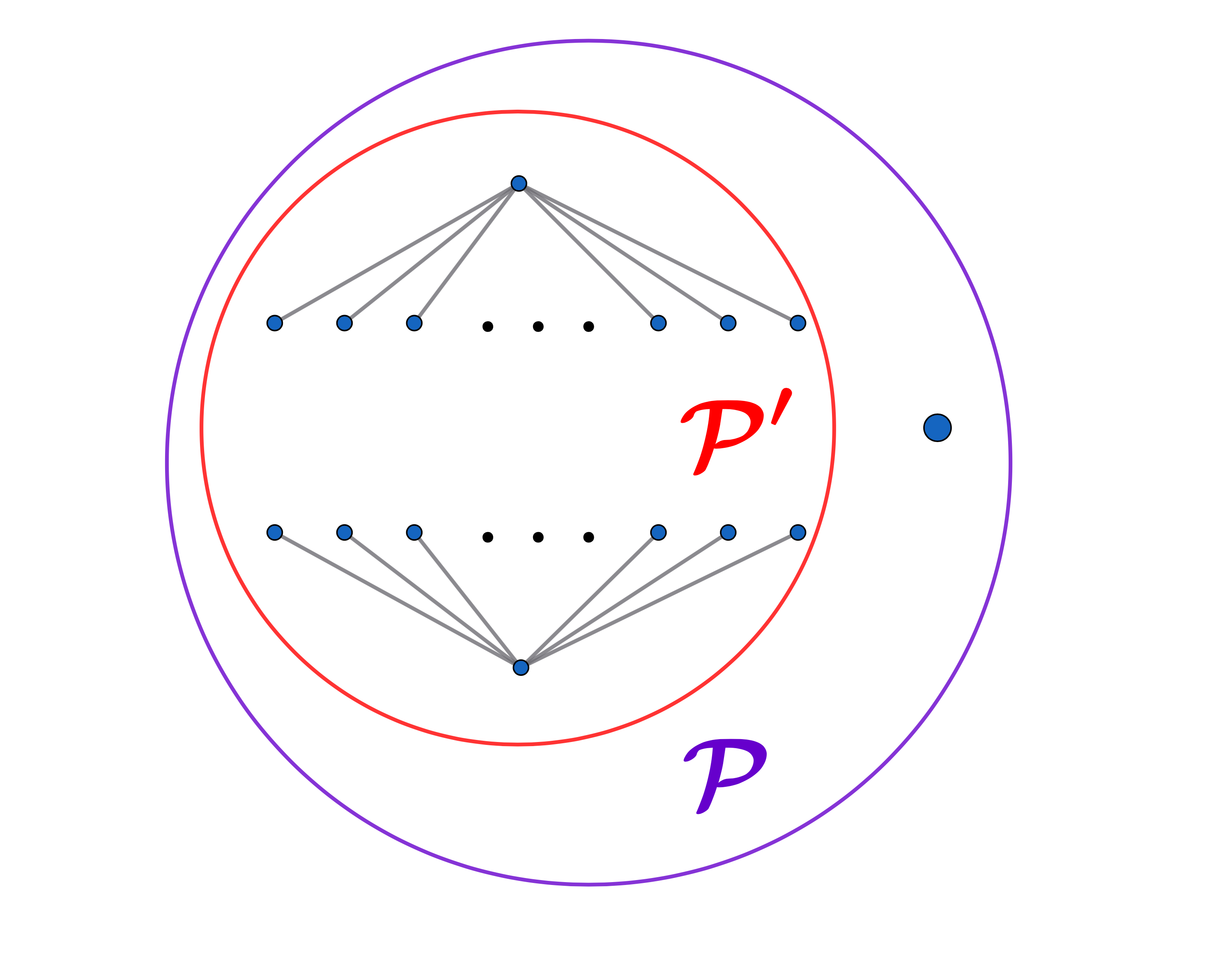}
\end{center}

We show that if we glue two special posets, the saturation number of the resulting poset is controlled above by the saturation numbers of the original posets. In particular, this means that by gluing two special posets with bounded saturation number we obtain a poset with bounded saturation number. We also establish linear upper bounds for all special posets.

\begin{theorem} Let $\mathcal P$ be a special poset. Then $\text{sat}^*(n,\mathcal P)=O(n)$.
    \label{special-linear}
\end{theorem}
\begin{proof}
Let $\mathcal P'$ denote the poset obtained from $\mathcal P$ by removing the element that is incomparable to everything. Let $k$ be the smallest natural number such that $\mathcal P([k])$ contains an induced copy of $\mathcal P'$. Let $h$ be the smallest natural number such that $\mathcal P([h])$ contains an induced copy of $\mathcal P$. We will show that $\text{sat}^*(n,\mathcal P)\leq 2^k+\text{sat}^*(n-k-1,\mathcal P)$ for every $n$, which iterated gives a linear bound indeed.

Since $\mathcal P'$ has both a unique maximal and a unique minimal element, the copy of $\mathcal P'$ in $\mathcal P([k])$ must have these two elements represented by $[k]$ and $\emptyset$ respectively. In particular, $\mathcal P ([k])$ does not contain an induced copy of $\mathcal P$. We now observe that if we add the element $k+1$ to all sets of this copy, together with the singleton $\{k+2\}$, we obtain a copy of $\mathcal P$ in $\mathcal P([k+2])$. This gives that $k+1\leq h\leq k+2$.

Suppose that $\mathcal P([n])$ contains an induced copy of $\mathcal P$. Let $A$ be the unique minimal element and $B$ the unique maximal element of $\mathcal P'$. By definition we must have $|B\setminus A|\geq k$. Let $C$ be the element that is incomparable to everything in this copy of $\mathcal P$. Since $C$ is incomparable to both $A$ and $B$, $A$ cannot be the empty set, and $B$ cannot be $[n]$, thus $n\geq |B|+1\geq k+|A|+1\geq k+2$. Therefore we must have $h=k+2$.

Let $n$ be a natural number and let $\mathcal F'$ be a $\mathcal P$-saturated family with ground set $\{k+2,\dots,n\}$. We define $\mathcal F$ to be $\mathcal P([k+1])\cup\{A\cup[k+1]:A\in\mathcal F'\}$. We claim that $\mathcal F$ is $\mathcal P$-saturated with ground set $[n]$. 

First, suppose that $\mathcal F$ contains an induced copy of $\mathcal P$. Since $h=k+2$ and $\{A\cup[k+1]:A\in\mathcal F'\}$ is isomorphic to $\mathcal F'$, hence $\mathcal P$-free, there must exist two elements of this copy, say $X$ and $Y$, such that $X\in\mathcal P([k+1])$ and $Y\in\{A\cup[k+1]:A\in\mathcal F'\}\setminus\mathcal P([k+1])$. By construction, this implies that $X\subset Y$, hence neither can be the element that is incomparable to everything, which we will call $Z$. But now, since $Z\in\mathcal F$, either $Z\in\mathcal P([k+1])$ in which case $Z\subset Y$, or $Z\in\{A\cup[k+1]:A\in\mathcal F'\}$ in which case $X\subset Z$, a contradiction. Therefore $\mathcal F$ is $\mathcal P$-free.

Finally, let $A\in\mathcal P([n])\setminus\mathcal F$. If $A=A'\cup[k+1]$ for some $A'\subseteq\{k+2,\dots,n\}$, then $A'\notin \mathcal F'$, otherwise $A\in \mathcal F$ by construction. This gives a copy of $\mathcal P$ in $\mathcal F'\cup\{A'\}$, which subsequently gives a copy of $\mathcal P$ in $\mathcal F\cup\{A\}$. If $A$ is not of that form, then there exists $i\in[k+1]$ such that $i\notin A$. Moreover, since $A\notin\mathcal F$, there exists $x\geq k+2$ such that $x\in A$. Let $\mathcal F''$ be the family $\{\{i\}\cup F:F\subseteq \mathcal P([k+1]\setminus\{i\})\}$. We observe that $\mathcal F''\subset\mathcal F$ and that it is isomorphic to $\mathcal P([k])$. Thus, it contains an induced copy of $\mathcal P'$. Since all the sets of this copy contain $i$ and do not contain $x$, this copy of $\mathcal P'$ together with $A$ gives an induced copy of $\mathcal P$ in $\mathcal F\cup\{A\}$. Therefore $\mathcal F$ is $\mathcal P$-saturated.

We therefore have that $|\mathcal F|\leq 2^k+|\mathcal F'|$. If we take $\mathcal F'$ to have minimal size, i.e. $|\mathcal F'|=\text{sat}^*(n-k-1,\mathcal P)$, we get the claimed inequality:
$$\text{sat}^*(n,\mathcal P)\leq 2^k+\text{sat}^*(n-k-1,\mathcal P).$$
\end{proof}
We remark that one could have also taken $\mathcal F$ to be $\mathcal F'\cup\{A\cup\{k+2,\dots,n\}:A\in\mathcal P([k+1])\}$, which will be useful later on in the section. 

Moreover, since the only restrictions on a special poset are the incomparable extra element and the unique maximal and minimal elements of the rest, we obtain the following corollary.
\begin{corollary}
\label{sub-poset}
    For any given finite poset $\mathcal P$, there exists a finite poset $\mathcal P^*$ such that $|\mathcal P^*|\leq|\mathcal P|+3$, $\mathcal P^*$ contains $\mathcal P$ as an induced sub-poset, and $sat^*(n,\mathcal P^*)=O(n)$.
\end{corollary}
We are now ready to prove the main result of this section, namely that if $\mathcal P_1$ and $\mathcal P_2$ are two special posets, then $\text{sat}^*(n, \mathcal P_2*\mathcal P_1)$ is bounded above by a function of $\text{sat}^*(n, \mathcal P_1)$ and $\text{sat}^*(n,\mathcal P_2)$, otherwise independent of $n$.

The idea of the proof is to start with a small family that is $\mathcal P_2*\mathcal P_1$-free and then extend it by looking at two special families, namely the sets in $\mathcal P([n])$ that are above a copy of $\mathcal P_1$ or below a copy of $\mathcal P_2$ (in the original family), which we will then carefully saturate using saturated families for $\mathcal P_1$ and $\mathcal P_2$. 
\begin{theorem}
\label{bounded}
Let $\mathcal P_1$ and $\mathcal P_2$ be two special posets. There exist constants $c_1, c_2$ and $c_3$ depending on $\mathcal P_1$ and $\mathcal P_2$ such that: $$\text{sat}^*(n,\mathcal P_2*\mathcal P_1)\leq c_1\text{sat}^*(n, \mathcal P_1)+c_2\text{sat}^*(n,\mathcal P_2)+c_3.$$ In particular, if both $\mathcal P_1$ and $\mathcal P_2$ have bounded saturation number, then so does $\mathcal P_2*\mathcal P_1$.
\end{theorem}
\begin{proof}
To start with, let $n$ be large (to be chosen later), $h_1$ be the smallest number such that $\mathcal P([h_1])$ contains an induced copy of $\bullet*\mathcal P_1$, and $h_2$ the smallest number such that $\mathcal P([h_2])$ contains a copy of $\bullet*\overline{\mathcal P_2}$, where $\overline{\mathcal P_2}$ is the poset obtained from $\mathcal P_2$ by reversing all relations (incomparability relations are the same). Crucially, since a special poset does not have a minimal or a maximal element, and the property of being special is invariant under reversing relations, $h_1$ is in fact the smallest number such that $\mathcal P([h_1])$ contains an induced copy of $\mathcal P_1$, and $h_2$ is the smallest number such that $\mathcal P([h_2])$ contains an induced copy of $\mathcal P_2$.  

Our starting family is:
$$\mathcal F_1= \binom{[h_1+h_2-1]}{\leq h_1}\cup \left\{[n]\setminus A : A\in\binom{[h_1+h_2-1]}{\leq h_2}\right\}.$$
\begin{claim2} $\mathcal F_1$ does not contain a copy of $\mathcal P_2*\mathcal P_1$.
\end{claim2}
\begin{proof} Suppose that such a copy exists. 

\noindent\textbf{Case 1.} All sets in the $\mathcal P_1$ part are in $\binom{[h_1+h_2-1]}{\leq h_1}$.

Let $X\subseteq [h_1+h_2-1]$ be their union (which is not in $\mathcal P_1$ as $\mathcal P_1$ does not have a unique maximal element). We therefore have an induced copy of $\mathcal{P}_1$ in $\mathcal{P}(X)$, so $|X|\geq h_1$. Moreover, $X$ is a strict subset of all the sets that make up the $\mathcal P_2$ part of this copy ($\mathcal P_2$ does not have a unique minimal element). This implies that this copy of $\mathcal P_2$ lives in $\left\{[n]\setminus A:A\in\binom{[h_1+h_2-1]}{\leq h_2}\right\}$. Looking at these sets and their intersection, and then taking complements, we see that we must have a copy of $\bullet*\overline{\mathcal P_2}$ in some $\mathcal P([n_0])$ for some $n_0\leq h_2-1$, which is a contradiction. 

\noindent\textbf{Case 2.} At least one element of the $\mathcal P_1$ part of the copy is in $\left\{[n]\setminus A:A\in\binom{[h_1+h_2-1]}{\leq h_2}\right\}$.

This implies that the entire $\mathcal P_2$ part is in $\left\{[n]\setminus A:A\in\binom{[h_1+h_2-1]}{\leq h_2}\right\}$, which by duality with the previous case, is impossible. 
\end{proof}

Let us now define $C_1 = \{A \in \mathcal P([n]) : A \text{ is above a copy of }\mathcal P_1\text{ in } \mathcal F_1\}$, and $C_2=\{A \in \mathcal P([n]): A\text{ is below a copy of }\mathcal P_2\text{ in }\mathcal F_1\}$. It is straightforward to see, and sometimes convenient, that \break${C_1=\left\{A \in \mathcal P([n]) : \exists B\in \binom{[h_1+h_2-1]}{ h_1}, B\subseteq A\right\}}$ and $C_2=\left\{A:\exists B\in\binom{h_1+h_2-1}{h_2}, B\subseteq [n]\setminus A\right\}$.

For two families $\mathcal A\subset\mathcal B\subseteq\mathcal P([n])$ and a poset $\mathcal G$, we say that \textit{$\mathcal A$ is $\mathcal G$-saturated for $\mathcal B$}, if $\mathcal A$ does not contain an induced copy of $\mathcal G$, but for any $X\in\mathcal B\setminus\mathcal A$, $\mathcal A\cup\{X\}$ contains an induced copy of $\mathcal G$.

\begin{claim2} Let $\mathcal F$ be a family that extends $\mathcal F_1$ in such a way that $C_1\cap \mathcal F$ is  $\mathcal P_2*\bullet$-saturated for $C_1$, and $C_2\cap \mathcal F$ is $\bullet*\mathcal P_1$-saturated for $C_2$. Then $\mathcal F$ is $\mathcal P_2*\mathcal P_1$-saturated.
\end{claim2}

\begin{proof}
We begin by noticing that $C_1\cup C_2\cup \mathcal F_1 = \mathcal P([n])$. This is because if we have a set $X\notin\mathcal F_1$ such that no subset of $[h_1+h_2-1]$ of size $h_1$ is a subset of $X$, then $X=X'\cup X''$ where $X'\subseteq [h_1+h_2-1]$, $|X'|\leq h_1-1$ and $X''\cap[h_1+h_1-1]=\emptyset$. Thus $X\subseteq X'\cup\{h_1+h_2,\dots,n\}$, which implies that $[h_1+h_2-1]\setminus X'\subseteq [n]\setminus X$, and so $X\in C_2$.

We first show that $\mathcal F$ does not contain a copy of $\mathcal P_2*\mathcal P_1$.
Suppose such a copy exists in $\mathcal F$, and call it $\mathcal P_2*\mathcal P_1$ for simplicity. Let $A = \cap_{B \in \mathcal P_2} B$. Notice that $A\notin C_1$. Indeed, if $A\in C_1$, then $A$ has a copy of $\mathcal P_1$ below it. In particular, any $B\in\mathcal P_2$ has the same copy below it, so $B\in C_1$ for all $B\in\mathcal P_2$. Finally, as $A\in C_1$, $A'\subseteq A$ for some $A'\in\binom{[h_1+h_2-1]}{h_1}$. This means that $\mathcal P_2$, together with $A'$, forms a copy of $\mathcal P_2*\bullet$ in $\mathcal F\cap C_1$, a contradiction.

Similarly $C\notin C_2$, where $C= \cup_{B\in \mathcal P_1} B$. Since $C\subseteq A$, we must have that $C\not\in C_1$ and $B\not\in C_2$. Therefore $A,C\in \mathcal F_1\setminus(C_1\cup C_2)$. If $C\subseteq [h_1+h_2-1]$, then the copy of $\mathcal P_2$ that is below it must be also comprised of subsets of $[h_1+h_2-1]$, thus $|C|\geq h_1$, which implies that $C\in C_1$, a contradiction. Similarly $[n]\setminus A$ is not a subset of $[h_1+h_2-1]$. This is impossible since $A$ and $C$ are in $\mathcal F_1$ and $C\subseteq A$.

 Finally, let $X\notin\mathcal F$. In particular $X\notin\mathcal F_1$. Without loss of generality we may assume that $X\in C_1$. Since $\mathcal F\cap C_1$ is $\mathcal P_2*\bullet$-saturated for $C_2$, we must have a copy of $\mathcal P_2*\bullet$ in $\mathcal F\cap C_1\cup\{X\}$ which contains $X$. Note that $X$ cannot be the minimal point of this poset since $X\in C_1$ which means that $C_1$ has a copy of $\mathcal P_1$ in $\mathcal F_1$ below it, which together with the copy of $\mathcal P_2$ above it forms a copy of $\mathcal P_2*\mathcal P_1$ in $\mathcal F$, a contradiction. Nevertheless, the minimal point of this copy is in $C_1$, thus it has a copy of $\mathcal P_1$ in $\mathcal F_1$ below it, which together with the copy of $\mathcal P_2$ that $X$ is part of, gives us a copy of $\mathcal P_2*\mathcal P_1$ that $X$ is part of. This shows that $\mathcal F$ is $\mathcal P_2*\mathcal P_1$-saturated. 
\end{proof}

Let $\mathcal F_0$ be a minimal $\mathcal P_2$-saturated family with ground set $\{h_1+h_2,\dots,n\}$. Now, for a fixed $A\in\binom{h_1+h_2-1}{h_1}$, we  observe that the sets in $\mathcal F_1$ strictly above $A$ are all sets of the form $\{h_1+h_2,\dots,n\}\cup A\cup X$, for all $X\subseteq[h_1+h_2-1]\setminus A$. Denote by $\overline{A}$ the set $[h_1+h_2-1]\setminus A$. Consider the family $\mathcal F_A=\mathcal F_0\cup\{X\cup\{h_1+h_2,\dots,n\}:X\subseteq\overline{A}\}$, which is $\mathcal P_2$-saturated with ground set $\overline{A}\cup\{h_1+h_2,\dots,n\}$. For every such $A$, we add to $\mathcal F_1$ all sets of the form $A\cup F$, where $F\in\mathcal F_A$.  Call this family $\mathcal F'$. 

By construction, for a given $A\in\binom{[h_1+h_2-1]}{h_1}$, the only sets in $\mathcal F'$ above $A$ are $A\cup F$ for $F\in\mathcal F_A$. This is also because if $A,B\in\binom{[h_1+h_2-1]}{h_1}$, $A\neq B$ and $A\subseteq B\cup F$ for some $F\in\mathcal F_B$, then $F\notin\mathcal F_0$, so $A\subset B\cup F=B\cup X\cup\{h_1+h_2,\dots,n\}$ for some $X\in\overline{B}$. Therefore we have that $A\subseteq B\cup X$. Let $Y=(B\cup X)\setminus A$. Then $Y\in\overline{A}$ and $B\cup F=B\cup X\cup\{h_1+h_2,\dots,n\}=A\cup Y\cup\{h_1+h_2,\dots,n\}=A\cup F'$, where $F'\in\mathcal F_A$. Moreover, we have added at most $\binom{h_1+h_2-1}{h_1}(2^{h_2-1}+\text{sat}^*(n-h_1-1,\mathcal P_2))\leq\binom{h_1+h_2-1}{h_1}(2^{h_2-1}+\text{sat}^*(n,\mathcal P_2))$ elements to $\mathcal F_1$. Note also that everything we have added is in $C_1$, hence $\mathcal F'\cap C_2=\mathcal F_1\cap C_2$.

In exactly the same way, using now the fact that $\mathcal P_1$ is a special poset, we extend $\mathcal F'$ in $C_2$ only, leaving the $C_1$ part intact. We call this final family $\mathcal F$. In light of Claim 1, and the fact that $|\mathcal F|\leq |\mathcal F_1|+\binom{h_1+h_2-1}{h_1}(2^{h_2-1}+\text{sat}^*(n,\mathcal P_2))+\binom{h_1+h_2-1}{h_2}(2^{h_1-1}+\text{sat}^*(n,\mathcal P_1))=c_1\text{sat}^*(n, \mathcal P_1)+c_2\text{sat}^*(n,\mathcal P_2)+c_3$, it is now enough to prove the claim below (the final condition in Claim 1 is completely analogous).

\begin{claim2}
$\mathcal F\cap C_1$ is $\mathcal P_2*\bullet$-saturated for $C_1$.
\end{claim2}
\begin{proof}
First, suppose that $\mathcal F\cap C_1$ contains a copy of $\mathcal P_2*\bullet$. Let $A$ be the minimal element of this copy, which, without loss of generality, can be assumed to be in $\binom{[h_1+h_2-1]}{h_1}$. This means that $A$ is below a copy of $\mathcal P_2$. Since, by construction, the only elements of $\mathcal F\cap C_1$ above $A$ are of the form $A\cup F$ where $F\in\mathcal F_A$, $\mathcal F_A$ must contain a copy of $\mathcal P_2$. This is a contradiction since this family was chosen to be $\mathcal P_2$-saturated for all such $A$. Hence $\mathcal F\cap C_1$ is $\mathcal P_2*\bullet$-free.

Finally, let $B\in C_1\setminus(\mathcal F\cap C_1)$. Since $B$ is in $C_1$, there exists some $A\in \binom{[h_1+h_2-1]}{h_1}$ such that $A\subset B$. Note in particular that $A\in\mathcal F\cap C_1$. Let $C=B\setminus A$ and observe that $C\notin\mathcal F_A$ as otherwise $B=A\cup C\in\mathcal F\cap C_1$, a contradiction. Therefore $\mathcal F_A\cup\{C\}$ contains a copy of $\mathcal P_2$, thus $\{A\cup F:F\in\mathcal F_A\}\cup\{B\}$ contains a copy of $\mathcal P_2$, which together with $A$ gives a copy of $\mathcal P_2*\bullet$ in $(\mathcal F\cap C_1)\cup\{B\}$, finishing the proof of the claim.
\end{proof}
\end{proof}
We remark that there exists an infinite family of special posets with bounded saturation number as shown in \cite{keszegh2021induced} (Theorem 3.8), namely the ones where $\mathcal P'$ is a chain of length at least 2. Thus Theorem~\ref{bounded} tells us that gluing any two such posets will result in a poset with bounded saturation number.
\section{Linear upper bound for $k$-layered complete posets}
Let $K_{n_1,\dots,n_k}$ denote the complete poset with $k$ layers of sizes $n_1,\dots,n_k$, in this order, where $n_1$ is the size of the bottom layer. The Hasse diagram is displayed below.
\begin{center}
\includegraphics[width=8.5cm]{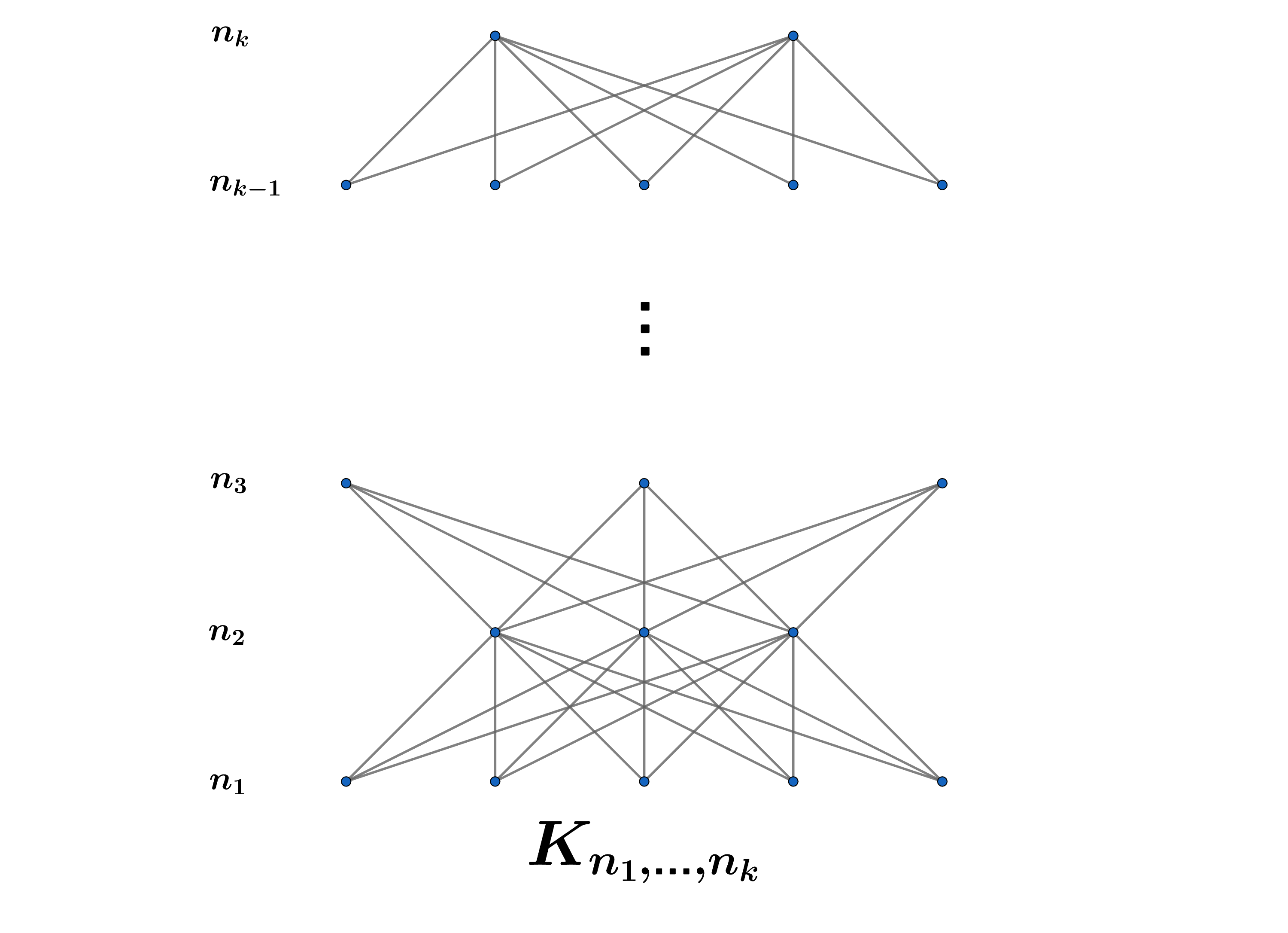}
\end{center}
In this section we establish linear upper bounds for all such posets that do not have two consecutive layers of size 1, i.e. there is no $i$ such that $n_i=n_{i+1}=1$. In light of Proposition~\ref{poset-point-poset}, it is enough to establish such an upper bound in the case $n_i\geq2$ for all $i$. The overview of the proof is as follows. We start with a family $\mathcal{F}_1\cup\mathcal{F}_2$ of subsets of $[n]$ of bounded size that is $K_{n_1,\dots,n_k}$-free. Moreover, $\mathcal F_1$ and $\mathcal F_2$ live in two disjoint and somehow incompatible copies of $Q_d$ (for some carefully chosen $d$), and they are each a union of $k-1$ layers of these hypercubes. We then show that regardless of how we saturate this family, we end up with a $K_{n_1,\dots,n_k}$-saturated family of size at most linear in $n$.

The idea behind the construction is that the maximum number of layers we can take in a hypercube without creating a chain of size $k$ is $k-1$, and then, in order to still ensure this property, the other hypercube has to be more or less the complement of the previous. This stops us, once we `exit' the first family, to reach the second family as we have been `pushed' too high.
\begin{theorem}\label{k-partite theorem}
Let $n_1,\dots,n_k\geq 2$ . Then, $sat^*(n, K_{n_1,\dots,n_k}) = O(n)$.
\end{theorem}
\begin{proof}
Before exhibiting the starting family, we need some notation. For a positive integer $n\geq2$, let $w(n)$ denote the smallest integer such that $\mathcal P([w(n)])$ contains an antichain of size $n$. In other words, $w(n)$ is the unique integer satisfying $\binom{w(n)}{\left\lfloor (w(n)/2)\right\rfloor}\geq n > \binom{\, w(n)-1}{\left\lfloor(w(n)-1)/2\right\rfloor}$. We also define $h(n)= \left\lfloor \frac{w(n)-1}{2}\right\rfloor$, and let $x(n)$ the smallest integer satisfying $\binom{x(n)+h(n)}{h(n)}\geq n$.

With this notation, given $n_1,\dots,n_k$ the sizes of the layers of our poset, we denote by $w_i=w(n_i), h_i=h(n_i)$ and $x_i=x(n_i)$ for all $1\leq i\leq k$. Finally, let $d=\sum_{i=1}^{k}{(h_i+x_i)}-1$. We are now ready to define our family $\mathcal F_1\cup\mathcal F_2$ such that $\mathcal F_1$ lives in a `low' copy of $Q_d$, while $\mathcal F_2$ lives in a disjoint `high' copy of $Q_d$. From each copy we will take a collection of full layers with spacing controlled by $h_i$'s and $x_i$'s. More precisely:
$$\mathcal F_1= \binom{[d]}{h_1}\cup \binom{[d]}{h_1+x_1+h_2}\cup\dots\cup \binom{[d]}{\sum_{i=1}^{k-3}(h_i+x_i)+h_{k-2}}\cup\binom{[d]}{\sum_{i=1}^{k-2}(h_i+x_i)+h_{k-1}},$$
and
$$\mathcal F_2=\left\{[n]\setminus A : A \in \binom{[d]}{h_k}\cup \binom{[d]}{h_k+x_k+h_{k-1}}\cup\dots\cup\binom{[d]}{\sum_{i=1}^{k-2}(h_{k+1-i}+x_{k+1-i})+h_{2}}\right\}.$$
The two families are represented in red in the figure below.
\begin{center}
    \includegraphics[width=10cm]{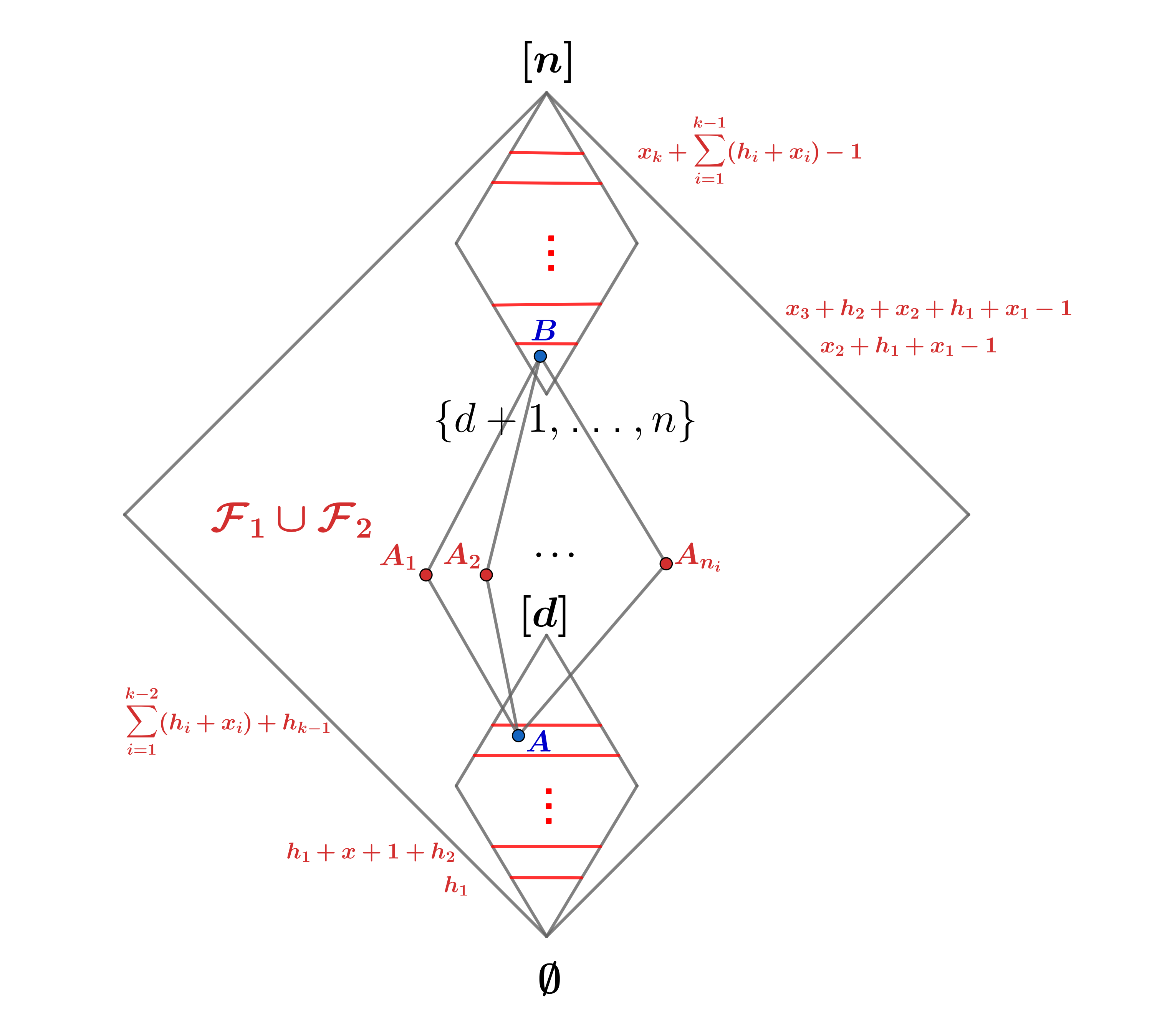}
\end{center}
\begin{proposition}
Let $\mathcal F$ be a $K_{n_1,\dots n_k}$-saturated family of sets of $[n]$ that contains $\mathcal F_1\cup\mathcal F_2$. Then $|\mathcal F|=O(n)$.
\end{proposition}
\begin{proof}
We start by defining $C_T$ to be the collection of sets in $\mathcal{F}$ that are exactly $T$ on $[d]$. In other words $C_T=\{F \in\mathcal F : F \cap [d] = T\}$, for all $T\subseteq [d]$.

Let $R\subseteq S$ be two sets such that $R\in\binom{[d]}{\sum_{j=1}^{i-1} (h_j+x_j)}$ and $S\in \binom{[d]}{\sum_{j=1}^{i} (h_j+x_j)-1}$ for some $i\in[k]$. We define $F_{R,S}=\cup_{R\subseteq T \subseteq S} C_T$. We now observe that the maximum size of an antichain in $F_{R,S}$ is at most $n_i-1$, thus, by Dilworth, $F_{R,S}$ is contained in $n_i-1$ chains.

Indeed, suppose $F_{R,S}$ contains an antichain of size $n_i$, and let $Y_1,\dots,Y_{n_i}$ be $n_i$ pairwise incomparable elements in $F_{R,S}$. Then $R\subseteq Y_1,\dots,Y_{n_i}\subseteq S\cup\{d+1,\dots,n\}$. However, by construction, $\mathcal P(R)\cap \mathcal F_1$ contains a copy  of $K_{n_1,\dots,n_{i-1}}$ below $R$. Also, by symmetry, there is a copy of $K_{n_{i+1},\dots,n_k}$ above $S\cup\{d+1,\dots,n\}$. These two copies, together with the antichain $Y_1,\dots, Y_{n_i}$ form a copy of $K_{n_1,\dots,n_k}$, a contradiction.

To finish the proof, we see that $\bigcup_{T\subseteq [d]}C_T = \mathcal{F}$. Since every $T\in[d]$ is between two sets $R\subseteq S$ with $R\in\binom{[d]}{\sum_{j=1}^{i-1} (h_j+x_j)}$ and $S\in \binom{[d]}{\sum_{j=1}^{i} (h_j+x_j)-1}$ for some $i\in[k]$, $$\bigcup _{T\subseteq [d]}C_T=\bigcup_{1\leq i \leq k}\bigcup_{(R,S)\in\mathcal A_i}F_{R,S},$$ where $\mathcal A_i=\left\{(X,Y):X\subseteq Y, X\in \binom{[d]}{\sum_{j=1}^{i-1} (h_j+x_j)}, Y\in \binom{[d]}{\sum_{j=1}^{i} (h_j+x_j)-1}\right\}.$

Therefore $\mathcal{F}$ is contained in the union of at most $\lambda$ chains, where $$\lambda=\sum_{i=1}^k (n_i-1)\binom{d}{\sum_{j=1}^{i-1} (h_j+x_j)}\binom{d -\sum_{j=1}^{i-1} (h_j+x_j) }{h_i+x_i-1},$$ and so $|\mathcal F|\leq \lambda (n-1)+2$.
\end{proof}
\begin{proposition}
The family $\mathcal{F}_1\cup\mathcal F_2$ is $K_{n_1,\dots,n_k}$-free.
\end{proposition}
\begin{proof}
Suppose $\mathcal F_1\cup \mathcal F_2$ contains such a copy, which, for simplicity, we call $K_{n_1,\dots,n_k}$. Since the longest chain that is entirely contained in one of $\mathcal F_1$ and $\mathcal F_2$ has size $k-1$, our copy must intersect both $\mathcal F_1$ and $\mathcal F_2$.

Let $i$ be the smallest integer such that the $i^{\text{th}}$ layer of $K_{n_1,\dots,n_k}$ contains an element of $\mathcal F_2$ (such an $i$ exists by the above observation). Observe first that $i$ cannot be 1 as this would imply $\mathcal F_2$ contains a chain of size $k$, a contradiction. Thus $1<i\leq k$. Suppose for now that $i<k$.

Let $A$ be the union of all the elements of the ${(i-1)}^{\text{th}}$ layer of $K_{n_1,\dots,n_k}$, and let $B$ be the intersection of all the elements of the $(i+1)^{\text{th}}$ layer of $K_{n_1,\dots,n_k}$. Since every element of the first $i-1$ layers of $K_{n_1,\dots,n_k}$ is in $\mathcal F_1$, we have that $A\subseteq [d]$. Since every element of the $i^{\text{th}}$ layer is less than every element if the $(i+1)^{\text{th}}$ layer, we must have that every element of the $i^{\text{th}}$ layer is strictly less than $B$ (as $n_i>1)$. By assumption, there is an element of $\mathcal F_2$ in the $i^{\text{th}}$ layer of $K_{n_1,\dots,n_k}$. This means that $n\in B$, which immediately implies that the layers $i+1,\dots,k$ are all entirely contained in $\mathcal F_2$. Consequently, $[n]\setminus[d]\subseteq B$. The sets $A$ and $B$ are represented in the above picture in blue.

Observe now that the $i^{\text{th}}$ level of $K_{n_1,\dots,n_k}$ is an antichain of size $n_i$, in $\mathcal F_1\cup \mathcal F_2$, such that every element of it is strictly greater than $A$ and strictly less than $B$.
\begin{claim1}$|A|\geq\sum_{j=1}^{i-1}(h_j+x_j)$.\end{claim1}
\begin{proof} We will prove something more general. We will show that for every $1\leq r\leq i$, if $C\subseteq [n]$ is a set such that there is an induced copy of $K_{n_1,\dots,n_{r-1}}$ in $\mathcal F_1$ strictly below it, then $|C|\geq \sum_{j=1}^{r-1}(h_j+x_j)$. The proof is by induction on $r$. When $r=1$ there is nothing to show. When $r=2$, the antichain of size $n_1$ must be in $\binom{C}{h_1}$, so $\binom{|C|}{h_1}\geq n_1$. By the definition of $x_1$, this means that $|C|\geq h_1+x_1$.

Suppose now that the result is true for $2\leq r-1<i$, and let $C$ be a set strictly above an induced copy of $K_{n_1,\dots,n_{r-1}}$ in $\mathcal F_1$. Let $C^*$ be the union of the first $r-2$ layers of this induced copy of $K_{n_1,\dots,n_{r-1}}$. By the induction hypothesis, $|C^*|\geq \sum_{j=1}^{r-2}(h_j+x_j)$. 

Let $X_1,\dots,X_{n_{r-1}}$ be the elements of the $(r-1)^{\text{th}}$ layer of this induced copy of $K_{n_1,\dots,n_{r-1}}$. Then $|X_t|>|C^*|$ for all $1\leq t\leq n_{r-1}$ (since every layer is an antichain of size at least 2). This means that the next level of the posets must jump up in $\mathcal F_1$, thus $|X_t|\geq \sum_{j=1}^{r-2}(h_j+x_j)+h_{r-1}\) for all $1\leq t\leq n_{r-1}$.

Observe that if $|X_t|>\sum_{j=1}^{r-2}(h_j+x_j)+h_{r-1}\) for some $t$, then $X_t$ must be in a higher level in $\mathcal F_1$, thus $|X_t|\geq \sum_{j=1}^{r-1}(h_j+x_j)+h_r$. Since $X_t\subset C$, we have $|C|\geq \sum_{j=1}^{r-1}(h_j+x_j)+h_r>\sum_{j=1}^{r-1}(h_j+x_j)$, as claimed.

Assume now that $|X_t|=\sum_{j=1}^{r-2}(h_j+x_j)+h_{r-1}\) for all $t$, and look at the sets $X_1\setminus C^*,\dots,X_{n_{r-1}}\setminus C^*$. They form an antichain of size $n_{r-1}$, and are all subsets of $C\setminus C^*$. Since $|X_t\setminus C^*|=h$ for some $h\leq h_{r-1}$ for all $t$, we get that $\binom{C\setminus C^*}{h}$ contains an antichain of size $n_{r-1}$, thus $|C\setminus C^*|\geq w_{r-1}$. But now we know that $\binom{C\setminus C^*}{h_{r-1}}$ must also contain an antichain of size $n_{r-1}$, therefore, $|C\setminus C^*|\geq h_{r-1}+x_{r-1}$, and so $|C|\geq\sum_{j=1}^{r-1}(h_j+x_j)$.

This finishes the inductive step, and thus the claim.
\end{proof}

By looking at the family $\overline{\mathcal F_2}=\{[n]\setminus A:A\in\mathcal F_2\}$, and applying the same argument as above, we also get that $|[n]\setminus B|\geq \sum_{j=i+1}^{k}(h_{j}+x_j)$. 

Observe now that $i$ cannot be equal to $k$ either, since this would imply that $|A|\geq\sum_{j=1}^{k-1}(h_j+x_j)$, and that the entire $k^{\text{th}}$ level of $K_{n_1,\dots,n_k}$ is in $\mathcal F_2$. Let $B^*$ be the intersection of all the elements of this level. Then we get $[n]\setminus[d]\subset B^*$ and $A\subset B^*$ by the same arguments as above. Moreover, the above result applies and we also have $|[n]\setminus B^*|\geq h_k+x_k$, which implies that $n\geq|\{d+1,\dots,n\}|+|A|+|[n]\setminus B^*|\geq n-d+\sum_{i=1}^{k-1}(h_j+x_j)+h_k+x_k=n+1$, a contradiction.

Let $A_1,\dots,A_{n_i}$ be the elements of the $i^{\text{th}}$ layer of $K_{n_1,\dots,n_k}$. They form an antichain, and each one of them is strictly less than $B$ and strictly greater that $A$.
\begin{claim1} Let $1\leq t\leq n_i$. If $A_t\in \mathcal F_1 $, then $A_t\in \binom{[d]}{\sum_{j=1}^{i-1}(h_j+x_j)+h_i}$.
\end{claim1}
\begin{proof} We know that $A\subseteq A_t\subseteq B\cap [d]$. Therefore, $|A|\leq |A_t|\leq |B\cap [d]|$. Since $B\subseteq [n]\setminus[d]$, we have that $|B\cap [d]| = d - |B^c|\leq\sum_{j=1}^{i}(h_j+x_j)-1 $. Therefore, $\sum_{j=1}^{i-1}(h_j+x_j) \leq |A_t|\leq \sum_{j=1}^{i}(h_{j}+x_j)-1$. Since $A_t$ is in $\mathcal F_1$, the only possible size in can have is $\sum_{j=1}^{i-1}(h_j+x_j) + h_i$, thus $A_t \in \binom{[d]}{\sum_{j=1}^{i-1}(h_j+x_j) + h_i}$, as claimed.
\end{proof}

Similarly, by taking complements, we also have that if $A_t\in\mathcal F_2$ for some $t\in[n_{r-1}]$, then $A_t\in \left\{[n]\setminus D : D \in \binom{[d]}{h_i+\sum_{j=i+1}^{k}(h_{j}+x_{j})}\right\}$. We therefore conclude that our $i^{\text{th}}$ layer, $A_1,\dots,A_{n_i}$, lies in $\binom{[d]}{\sum_{j=1}^{i-1}(h_j+x_j)+h_i}\cup\left\{[n]\setminus D : D \in \binom{[d]}{h_i+\sum_{j=i+1}^{k}(h_{j}+x_{j})}\right\}$.

Let us write $B=\{d+1,\dots,n\}\cup A\cup M$, where $M$ is a subset of $[d]$ and $A\cap M=\emptyset$.
\begin{claim1} Let $t_1$ and $t_2$ be two distinct integers in $[n_i]$. Then $A_{t_1}\cap M$ and $A_{t_2}\cap M$ are incomparable.
\end{claim1}
\begin{proof}
If $A_{t_1}$ and $A_{t_2}$ are both in $\mathcal F_1$, then they are subsets of $\binom{[d]}{\sum_{j=1}^{i-1}(h_j+x_j)+h_i}$, so $A_{t_1}\cap M=A_{t_1}\setminus A$ and $A_{t_2}\cap M=A_{t_2}\setminus A$, thus $A_{t_1}\cap M$ and $A_{t_2}\cap M$ are incomparable as $A_{t_1}$ and $A_{t_2}$ are incomparable and $A$ is a subset of both. Similarly, if they are both in $\mathcal F_2$, then $\{d+1,\dots,n\}\cup A\subset A_{t_1}, A_{t_2}$, thus $A_{t_1}\cap M=A_{t_1}\setminus{(\{d+1,\dots,n\}\cup A)}$ and $A_{t_2}\cap M=A_{t_2}\setminus\{d+1,\dots,n\}\cup A$, which are again incomparable. Hence we may assume that $A_{t_1}\in\binom{[d]}{\sum_{j=1}^{i-1}(h_j+x_j)+h_i}$ and $A_{t_2}\in\left\{[n]\setminus D: D\in\binom{[d]}{h_i+\sum_{j=i+1}^k(h_j+x_j)}\right\}$.

Let us write $A_{t_1}=A\cup T_{1}$ and $A_{t_2}=\{d+1,\dots,n\}\cup A\cup T_{2}$, where $T_1$ and $T_2$ are subsets of $[d]$ disjoint from $A$. Therefore we have $A_{t_1}\cap M=T_1$ and $A_{t_2}\cap M=T_2$. We cannot have $T_1\subseteq T_2$ as this would imply $A_{t_1}\subset A_{t_2}$. If $T_2\subsetneq T_1$, then $|T_2|+1\leq|T_1|$. However, the size of $T_1$ is $|A_{t_1}|-|A|$ and $|T_2|=|A_{t_2}|-n+d-|A|$, and so $\sum_{j=1}^{i-1}(h_j+x_j)+h_i\geq n-h_i-\sum_{j=i+1}^k(h_j+x_j)-n+d+1.$ Rearranging this we get that $x_i\leq h_i$, meaning that $\binom{2h_i}{h_i}\geq n_i$, which consequently implies that $w_i\leq2h_i$, a contradiction. 
\end{proof}

Therefore $A_1\cap M,\dots,A_{n_i}\cap M$ form an antichain. If $A_t\in\binom{[d]}{\sum_{j=1}^{i-1}(x_j+h_j)+h_i}$, we know that $A\subset A_t\subset\{d+1,\dots,n\}\cup A\cup M$, thus $A\subset A_t\subseteq A\cup M$, thus $|A_t\cap M|=|A_t|-|A|=h'\leq h_i$ for all $t\in[n_i]$. Similarly we obtain that if $A_i\in\left\{[n]\setminus D : D \in \binom{[d]}{h_i+\sum_{j=i+1}^{k}(h_{j}+x_{j})}\right\}$, then $A_i\cap M$ is $M\setminus D$ for some $D\in[d]$ of size $h''=|[n]\setminus A_t|-|[n]\setminus B|\leq h_i$. Therefore, this antichain lives inside $\binom{M}{h'}\cup \left\{M\setminus D : D \in \binom{M}{h''}\right\}$ for some $h', h''\leq h_i$.

Let $a_1$ be the number of sets of this antichain that are in $\binom{M}{h'}$, and $a_2$ the number of sets of the antichain that are in $\{M\setminus D:D\in\binom{M}{h''}\}$. Then the LYM inequality tells us that $\frac{a_1}{\binom{|M|}{h'}}+\frac{a_2}{\binom{|M|}{h''}}\leq 1$. Assume without loss of generality that $\binom{|M|}{h''}\leq\binom{|M|}{h'}$, thus $a_1+a_2=n_i\leq\binom{|M|}{h'}$. Since $|M|=d-|A|-|[n]\setminus B|\leq h_i+x_i-1$, we have that $n_i\leq\binom{x_i+h_i-1}{h'}$. By the definition of $w_i$, we get that $w_i\leq x_i+h_i-1$. But $2h_i<w_i$, and so $2h_i<x_i+h_i-1$, which implies that $h_i\leq\lfloor\frac{x_i+h_i-1}{2}\rfloor$. Since $h'\leq h_i$ we get that $n_i\leq\binom{x_i+h_i-1}{h_i}$, which contradicts the definition of $x_i$. 

Therefore such a copy cannot exist, i.e. $\mathcal F_1\cup\mathcal F_2$ is $K_{n_1,\dots,n_k}$-free.
\end{proof}
\end{proof}
\begin{corollary}
Let $k\geq2$ and $n_1, n_2,\dots,n_k$ be natural numbers with the property that for no $i\in[k]$ we have $n_i=n_{i+1}=1$. Then $\text{sat}^*(n,K_{n_1,\dots,n_k})=O(n)$.    
\end{corollary}
\begin{proof} The proof is a simple application of Theorem~\ref{k-partite theorem} and Proposition~\ref{poset-point-poset}. More precisely, starting with the layers that have size at least 2, we repeatedly apply the fact that $\text{sat}^*(n,\mathcal{P}_2*\mathcal{P}_1)\geq\text{sat}^*(n,\mathcal P_2*\bullet*\mathcal P_1)$ whenever $\mathcal P_2$ does not have a unique minimal element and $\mathcal P_1$ does not have a unique maximal element, inserting at each step the layers of size 1.
\end{proof}
\section{Poset percolating families}
In this section we discuss a weak notion of saturation for posets. Let $\mathcal P$ be a finite poset and $n$ a natural number. We say that a family $\mathcal F\subseteq\mathcal P([n])$ \textit{is percolating for $\mathcal P$} if there exists an ordering $A_1,A_2,\dots A_N$ of $\mathcal P([n])\setminus\mathcal F$, such that, for all $i\in[N]$, $\mathcal F\cup\{A_1,\dots A_i\}$ contains an induced copy of $\mathcal P$ in which $A_i$ is an element. This is equivalent to saying that, starting with $\mathcal F$, we can exhaust $\mathcal P([n])$ by adding at each step all the elements that form a new induced copy of $\mathcal P$.

We define \textit{the percolation number of $\mathcal P$}, denoted by $\text{sat}_p(n,\mathcal P)$, to be the minimum size of a $\mathcal P$-percolating family. Since we are looking for a minimal $\mathcal P$-percolating family, it must be $\mathcal P$-free (otherwise we exclude one element from a copy of $\mathcal P$ and list it before everything else, hence obtaining a strictly smaller $\mathcal P$-percolating family). It is clear that any such family must have at least $|\mathcal P|-1$ sets, otherwise there is no copy of $\mathcal P$ at step 1. We show that indeed, for $n$ large enough, the percolation number for any poset is $|\mathcal P|-1$ if $\mathcal P$ has both a unique minimal and a unique maximal element, $|\mathcal P|+1$ if $\mathcal P$ does not have a unique maximal or a unique minimal element, and $|\mathcal P|$ otherwise. Note that trivially this is the smallest it can be in all the cases since, if $\mathcal P$ does not have a unique minimal (maximal) element, then $\emptyset$ ($[n]$) must be in the family as they are never part of an induced copy of $\mathcal P$.

\begin{theorem} Let $\mathcal P$ be a finite poset consisting of $p$ elements, and $n\geq 3p-1$ a positive integer. Then $\text{sat}_p(n,\mathcal P)=p+1$ if $\mathcal P$ does not have a unique maximal or a unique minimal element, $\text{sat}_p(n,\mathcal P)=p-1$ if $\mathcal P$ has both a unique maximal and a unique minimal element, and $\text{sat}_p(n,\mathcal P)=p$ otherwise.
\label{percolating}    
\end{theorem}
\begin{proof}
Let $\mathcal P=\{p_1,p_2,\dots, p_p\}$ with its partial order given by $\leq'$. We say that $X_1,\dots, X_p$ is the \textit{canonical copy of $\mathcal P$ on ground set $\{x_1<\dots<x_p\}$} if $x_i\in X_j\Leftrightarrow p_i\leq' p_j$. It is straightforward to check that this is indeed an induce copy of $\mathcal P$. 

Let $A_1,\dots,A_p$ be the canonical copy of $\mathcal P$ on $[p]$. If $\mathcal P$ does not contain a unique maximal or a unique minimal element, then we will take  $\mathcal F$ to be $\{A_2,\dots,A_p, \emptyset, [n]\}$, if $\mathcal P$ has both a unique maximal and a unique minimal element, then we will take $\mathcal F$ to be $\{A_2,\dots,A_p\}$, and if $\mathcal P$ has a unique minimal (maximal) element, but not a unique maximal (minimal) element, we take $\{A_2,\dots,A_p\}$ together with $[n]$ ($\emptyset$). We observe that, in order to show that $\mathcal F$ is percolating for $\mathcal P$, regardless of the structure of $\mathcal P$, it is enough to show that, starting with $\{A_1,A_2,\dots, A_p\}$, we can add one by one all sets $A\in\mathcal P([n])\setminus\{A_1,\dots,A_p,\emptyset, [n]\}$ in such a way that we create a new copy of $\mathcal P$ at every step.

Now let $A\subseteq[n]\setminus[p]$. Then, if we replace $A_j$ with $A\cup A_j$, where $A_j$ is a maximal element of this copy of $\mathcal P$, we still have an induced copy of $\mathcal P$. Thus, going in levels in the canonical copy of $\mathcal P$ we started with, from the maximal elements to the minimal elements, we can add to our family every set of the form $A\cup A_i$ for all $i\in[p]$ and all $A\subseteq [n]\setminus [p]$. Call this new family $\mathcal F_1$.

Let $A_k$ be a minimal element of this copy of $\mathcal P$. By construction $A_k=\{k\}$. We notice that the function $f:\{A_1,\dots,A_p\}\rightarrow\mathcal F_1\cup\{A\}$ given by $f(A_k)=A$, $f(A_i)=A\cup A_i$ if $k\in A_i$ and $i\neq k$, and $f(A_i)=A_i$ otherwise, is an order-preserving injection. This means that we can add any such $A$ to $\mathcal F_1$. Let this new family be $\mathcal F_2$. Since $\mathcal F_2$ contains $\mathcal P(\{p+1, p+2,\dots,n\})$, looking at the canonical copy of $\mathcal P$ on ground set $\{p+1,\dots,2p\}$ and running the same argument as above, we see that we can also add $\mathcal P([p])$ to $\mathcal F_2$. Call this new family $\mathcal F_3$.

We are left to show that we can add to $\mathcal F_3$ all sets $B\notin\{\emptyset,[n]\}$ of the form $B=A\cup X$, where $A\subseteq\{p+1,\dots,n\}$, $X\subseteq[p]$ and $A$ and $X$ are non-empty.

If $|A|\leq p-1$, then let $\{a_1,a_2,\dots,a_p\}\subseteq\{p+1,\dots,n\}\setminus A$, and look at the canonical copy of $\mathcal P$ on ground set $\{a_1,\dots,a_p\}$. As noted before, at least one set is a singleton (any minimal set), say $\{a_t\}$. We now replace, in all sets of this copy, the element $a_t$ with the set $A$. This gives us a copy of $\mathcal P$ in $\mathcal F_3$ in which $A$ is one of the minimal elements. Finally, by going down in layers from the maximal elements to the minimal elements of this copy, and adding $X$ to every set, we conclude that we can eventually add $A\cup X$ too.

If $|A|\geq p$ and $A\neq\{p+1,\dots,n\}$, then let $l\in\{p+1,\dots,n\}\setminus A$. Let also $\{b_1,b_2,\dots b_{p-1}\}\subset A$, and $\overline{\mathcal{P}}$ be the poset obtained from $\mathcal P$ by reversing all relations. Consider the canonical copy of $\overline{\mathcal P}$ with ground set $S=\{l,b_1,b_2,\dots b_{p-1}\}$. By taking the complements in $S$ of the sets that make up this copy, we obtain a copy of $\mathcal P$ where all the maximal elements are complements of some singletons in $S$. After potentially reordering, we may assume that one such set is $\{b_1,b_2,\dots,b_{p-1}\}$, which can now be replaced by $A\cup X$. Thus we can add $A\cup X$ to $\mathcal F_3$.

Finally, let $A=\{p+1,\dots,n\}$ and $X\neq [p]$. By the above we may assume that any set that is not of this form has been added to the family. Let $k\in[p]\setminus X$ and, similarly as before, construct a copy of $\mathcal P$ on ground set $T=\{k,p+1,\dots,2p-1\}$ whose maximal elements are all complements of singletons in $T$. After reordering, if necessary, we may assume that one such maximal element is $\{p+1,\dots,2p-1\}$, which can be replaced by $A\cup X$, thus finishing the proof. 
\end{proof}
\section{Conclusions and further work}
The gluing operation seems to be playing a significant role in our understanding of posets at their saturation number -- in fact it is the only know operation on the set of finite posets that follows any type of monotone property (in the sense that if a poset is an induced copy of another poset, then the saturation number of one controls the saturation number of the other).

There is, however, one condition that we often have to impose on the posets in order to make our proofs work, namely that they do not have a unique minimal or a unique maximal element. This suggests to us that the gluing operation and Conjecture~\ref{maxelement} are tightly related. Another consequence of this restriction is the fact that we have been unable to obtain linear upper bounds for complete posets that have at least two consecutive layers of size 1. For example, what happens to the saturation number if we take two complete posets, both with all layers non-trivial, and glue them together via a long chain? In other words, can $\text{sat}^*(n,K_{n_1,\dots, n_k}*C_M*K_{n'_1,\dots,n'_{k'}})$ be of superlinear growth for some very large $M$? We believe that this is not true. In fact, we conjecture the following.
\begin{conjecture}
Let $n_1,\dots,n_k$ be positive integers such that at least one of them is not equal to 1. Then $\text{sat}^*(n, K_{n_1,\dots,n_k})=\Theta(n)$.
\end{conjecture}
Returning to the gluing operation, note that, in the unbounded case, we have managed to show that the saturation number of a poset obtained via this operation is unbounded if at least one of the posets has the UCTP. We in fact believe that the following stronger statement should hold.
\begin{conjecture}
Let $\mathcal P_1$ and $\mathcal P_2$ be two finite posets. Then $\text{sat}^*(n,\mathcal P_2*\mathcal P_1)$ is unbounded if and only if at least one of $\mathcal P_1$ or $\mathcal P_2$ has unbounded saturation number.   
\end{conjecture}
This is a strengthening of Conjecture~\ref{maxelement}. We would also like to mention that the proof of Theorem~\ref{bounded} relies heavily on the existence of bounded saturated families that contain a certain hypercube. For general posets with bounded saturation number these families may not exist. Nevertheless, we think that the result should hold in full generality.

\begin{conjecture}
Let $\mathcal P_1$ and $\mathcal P_2$ be two finite posets with bounded saturation number. Then $\mathcal P_2*\mathcal P_1$ also has bounded saturation number.
\end{conjecture}

Finally, we mention a surprising observation, that ties \textit{all} complete multipartite posets with no trivial layers and at least one layer of size 2, to the famous diamond poset. By applying Proposition \ref{propositionmax} at most twice, and potentially using the fact that $n+1=\text{sat}^*(n,\Lambda_2)=\text{sat}^*(n, \mathcal A_2)=\text{sat}^*(n,\mathcal V_2)\geq\text{sat}^*(n,\mathcal D_2)$, we get that $\text{sat}^*(n,K_{n_1,\dots,n_k})\geq\text{sat}^*(n,\mathcal D_2)$, as long as $n_i=2$ for some $i\in[k]$ and $n_t\geq2$ for all $t\in[k]$. This says that establishing a linear lower bound for the diamond would immediately imply that the saturation number for all these complete posets is also linear, showing once again how crucial in the poset saturation area the diamond really is. 

\vspace{2em}
\textbf{Acknowledgment.} The first author would like to thank the Institute for Advanced Study, Princeton, where most of this paper was written.

\vspace{1em}
\textbf{Note added in proof.} We want to thank Igor Pak who brought to our attention the fact that the `gluing' operation considered in this paper appears in literature as the \textit{linear sum} of posets.
\bibliographystyle{amsplain}
\bibliography{references}
\Addresses
\end{document}